\documentclass[12pt,a4paper,openbib]{article}
\usepackage[utf8]{inputenc}
\usepackage[T1]{fontenc}
\usepackage[french]{babel}
\usepackage{amsmath}
\usepackage{amsfonts}
\usepackage{dsfont}
\usepackage{tikz}
\usetikzlibrary{cd}
\usepackage{authblk}
\usepackage{amssymb}
\usepackage{amsthm}
\usepackage{lmodern}
\usepackage{fancyhdr}
\usepackage{lettrine}
\usepackage{bm}
\usepackage{enumerate}
\usepackage{multicol}
\usepackage[hidelinks]{hyperref}
\usepackage[nottoc, notlof, notlot]{tocbibind}
\usepackage[left=2cm,right=2cm,top=2cm,bottom=2cm]{geometry} 
\usepackage{graphicx}
\usepackage{mathtools}
\usepackage{wasysym}
\usepackage{stmaryrd}
\usepackage[all,cmtip]{xy}
\usepackage{enumitem}
\usepackage{relsize} 
\usepackage{mathrsfs}

\makeatletter

\newcommand{\address}[1]{\gdef\@address{#1}}
\newcommand{\email}[1]{\gdef\@email{\url{#1}}}
\newcommand{\@endstuff}{\par\vspace{\baselineskip}\noindent\small
\begin{tabular}{@{}l}\scshape\@address\\\textit{E-mail address:} \@email\end{tabular}}
\AtEndDocument{\@endstuff}
\makeatother

\author{Damien Junger\footnote{This work has been written in a great part during the author PhD thesis at the ENS Lyon. His work are currently funded by the Deutsche Forschungsgemeinschaft (DFG, German Research Foundation) under Germany's Excellence Strategy EXC 2044–390685587, Mathematics Münster: Dynamics–Geometry–Structure.}}
\address{Mathematisches Institut, Universität Münster,\\ Fachbereich Mathematik und Informatik der Universität Münster,  Orléans-Ring 10, 48149 Münster, Germany.}
\email{djunger@uni-muenster.de}
\title{Équations pour le premier revêtement de l'espace symétrique de Drinfeld}

\newtheorem{theointro}{Th\'eor\`eme}

\newtheorem{corointro}[theointro]{Corollaire}
\newtheorem{theo}{Th\'eor\`eme}[section]
\newtheorem{lem}[theo]{Lemme}
\newtheorem{coro}[theo]{Corollaire}
\newtheorem{prop}[theo]{Proposition}

\theoremstyle{definition}
\newtheorem{defi}[theo]{D\'efinition}

\theoremstyle{remark}
\newtheorem{rem}[theo]{Remarque}

\newtheorem{ex}[theo]{Exemple}

\newtheorem{notation}[theo]{Notation}

\DeclareMathOperator{\spec}{Spec}
\DeclareMathOperator{\spf}{Spf}
\DeclareMathOperator{\spg}{Sp}

\DeclareMathOperator{\gln}{GL}
\DeclareMathOperator{\pgln}{PGL}

\DeclareMathOperator{\modut}{-Mod}

\DeclareMathOperator{\hhh}{H}

\DeclareMathOperator{\en}{End}
\DeclareMathOperator{\aut}{Aut}

\DeclareMathOperator{\id}{Id}

\DeclareMathOperator{\nrm}{Nrm}

\DeclareMathOperator{\pic}{Pic}
\DeclareMathOperator{\ord}{ord}

\DeclareMathOperator{\lie}{Lie}

\DeclareMathOperator{\proj}{Proj}

 \newcommand{\iso}{\stackrel{\sim}{\fl}}

\font\tengoth=eufb10
\font\sevengoth=eufb7
\font\fivegoth=eufb5

\newfam\gothfam
\textfont\gothfam=\tengoth
\scriptfont\gothfam=\sevengoth
\scriptscriptfont\gothfam=\fivegoth

\def\B{{\mathbb{B}}}

\def\F{{\mathbb{F}}}
\def\G{{\mathbb{G}}}
\def\H{{\mathbb{H}}}

\def\P{{\mathbb{P}}}
\def\Q{{\mathbb{Q}}}

\def\Z{{\mathbb{Z}}}

\def\BC{{\mathcal{B}}}
\def\CC{{\mathcal{C}}}

\def\FC{{\mathcal{F}}}
\def\HC{{\mathcal{H}}}
\def\IC{{\mathcal{I}}}
\def\GC{{\mathcal{G}}}

\def\MC{{\mathcal{M}}}
\def\OC{{\mathcal{O}}}
\def\PC{{\mathcal{P}}}

\def\TC{{\mathcal{T}}}
\def\UC{{\mathcal{U}}}

\def\XG{{\mathfrak{X}}}

\def\mG{{\mathfrak{m}}}

\def\pG{{\mathfrak{p}}}

\def\Ff{{\mathscr{F}}}

\def\Lf{{\mathscr{L}}}

\def\Of{{\mathscr{O}}}
\def\Pf{{\mathscr{P}}}

\def\bar#1{\overline{#1}}
\def\het#1{{\rm H}^{#1}_{\rm ét}}

\def\hgal#1{{\rm H}^{#1}_{\rm Gal}}

\def\et{\text{ et }}
\def\si{\text{Si }}
\def\sinon{\text{Sinon }}
\def\and{\text{ and }}

\def\fl{\rightarrow}

\def\fln#1#2{\xrightarrow[#2]{#1}}

\def\flinj{\hookrightarrow}

\def\fla{\mapsto}

\def\limp{\varprojlim}

\def\som#1#2#3{\sum\limits_{{\substack{#2}}}^{#3}{#1}}
\def\pro#1#2#3{\prod\limits_{{\substack{#2}}}^{#3}{#1}}

\def\uni#1#2#3{\bigcup\limits_{{\substack{#2}}}^{#3}{#1}}

\begin{document}

\maketitle

\begin{abstract}
The goal of this work is to study some aspects of the geometry of the first cover $\Sigma^1$ in the Drinfeld tower over $\mathbb{H}^d_K$ the Drinfeld symmetric space over $K$ a finite extension of $\mathbb{Q}_p$. It is a cyclic étale cover of order prime to $p$ and even of Kummer type from the vanishing of the Picard group of $\mathbb{H}^d_K$ shown in a previous work of the author. It is then completely described by a certain class of invertible functions on $\mathbb{H}^d_K$ via the Kummer exact sequence and the main result of this article gives an explicit description of this class thus providing "equations" for $\Sigma^1$. This statement extends and uses crucially the local description over a vertex obtained by Wang (and originally by Teitelbaum in dimension 1). One of the main consequence of our global equation is the description of invertible functions of $\Sigma^1$ in terms of the invertible functions of $\mathbb{H}^d_K$.
\end{abstract}

\tableofcontents
 
\section*{Introduction}

Soit $p$ un nombre premier, 
$K$ une extension finie de $\Q_p$.
Nous nous proposons  d'étudier certains aspects de la géométrie de l'espace symétrique de Drinfeld $\H^d_K$ sur $K$ de dimension $d$ et de ses rev\^etements construits par Drinfeld dans l'article monumental \cite{dr2}. 
L'étude de la cohomologie de ces revêtements s'est révélée fondamentale pour établir et comprendre les correspondances de Jacquet-Langlands et de Langlands locale et a fait l'objet d'une littérature conséquente dont les articles \cite{dr1, cara3, harrtay, falt, fargu, dat1, dat0, mied, Scholze1} ... en constituent une liste non exhaustive.

Toutefois leur géométrie reste assez mystérieuse et est très mal comprise. Le résultat principal de cet article parvient tout de même à donner une description explicite au premier revêtement et fournit une équation globale à cet espace. Notons que la géométrie se complexifie grandement pour les autres revêtements mais elle reste raisonnable en niveau $1$. Par exemple, son groupe de galois est cyclique d'ordre premier à $p$ ce qui n'est pas le cas en niveau supérieur. Une autre propriété essentielle, uniquement valable en niveau $1$, provient de l'interprétation modulaire de\footnote{avec $\breve{K}= \widehat{K^{nr}}$ une compl\'etion de l'extension maximale non ramifi\'ee de $\bar{K}$.}  $\H_{\breve{K}}^d:=\H^d_K\hat{\otimes} \breve{K}$. Plus précisément, on a un modèle entier semi-stable\footnote{avec $\OC_{\breve{K}}$ l'anneau des entiers de $\breve{K}$.} $\H_{\OC_{\breve{K}}}^d$ qui représente un problème modulaire et on note $\XG/\H_{\OC_{\breve{K}}}^d$ le groupe formel universel associé. Ce dernier est muni d'une action de $\OC_D$ l'ordre maximal  de l'algèbre à division $D$ d'invariant $1/(d+1)$ sur $K$ et $\Sigma^1$ est la fibre générique de $\XG[\Pi_D]\backslash \{0\}$ où $\Pi_D$ est une uniformisante de $D$ et $\XG[\Pi_D]$ sont les points de $\Pi_D$-torsions de $\XG$. L'observation fondamentale consiste à voir que $\XG[\Pi_D]$ est un sch\'ema en $\F_p$-espace vectoriel de Raynaud qui est une   
classe de schéma en groupe dont on connait une classification \cite{Rayn}. Ce résultat permet de comprendre la fibre sp\'eciale de $\XG[\Pi_D]$ et fait l'objet des articles \cite{teit2} et \cite[sections 2.3, 2.4]{wa}. Nous nous en servons ici pour décrire $\Sigma^1$.

Avant d'énoncer le r\'esultat, nous introduisons quelques notations. 
Si $X$ est un espace rigide, $u$ une section inversible sur $X$ et $N$ un entier premier \`a $p$, on note \[X(u^{1/N}):=\underline{\spec}_X(\Of_X[T]/(T^N-u))\] le rev\^etement de $X$ de Kummer d'ordre $N$ associ\'e \`a $u$ que l'on peut voir comme une classe\footnote{En particulier, $\kappa :\Of_X^*(X)\to \het{1}(X,\mu_N)$ corresponds à la flèche de bord dans la suite exacte longue de Kummer.} $\kappa(u)\in\het{1}(X,\mu_N)$.

Rappelons que $\H_K^d$ est le complémentaire des hyperplans $K$-rationnels dans l'espace projectif rigide de dimension $d$. En particulier, on a un recouvrement admissible affinoide croissant\footnote{où $C=\widehat{\bar{K}}$ une compl\'etion d'une clôture alg\'ebrique de $K$.}   $\H_C^d(:=\H_K^d\hat{\otimes}C)= \bigcup_n \bar{U}_n$ (voire m\^eme Stein d'après \cite[§1, Proposition 4]{scst}) et
\[  \bar{U}_n =\{z\in\P^d_{ C} : |\varpi|^{-n} \ge |\frac{\tilde{a}_0z_0+\cdots \tilde{a}_dz_d}{\tilde{b}_0z_0+\cdots \tilde{b}_dz_d}|\ge |\varpi|^n,\forall a\neq b \in \P^d(\OC_K/ \varpi^{n+1} \OC_K)\} \] où $\tilde{a}$, $\tilde{b}$ sont des relevés $a,b\in \P^d(\OC_K/ \varpi^n \OC_K)$ dans $K^{d+1}$ par des vecteurs unimodulaires ie. $\max_i | a_i | =\max_i | b_i |=1$ (ne dépend pas du choix des relevés par ultramétrie).

Le résultat principal est le suivant : 

\begin{theointro}
\label{theointroeq}
Soit $N=q^{d+1}-1=(q-1)\tilde{N}$. On a \[\Sigma^1_C(:=\Sigma^1\hat{\otimes}C)= \H_C^d((u^{q-1})^{1/N})\] pour tout $u \in \Of^{*}(\H_C^d)$ tel que
\begin{equation}\label{eq:ucong}
 u|_{\bar{U}_n} \equiv  \prod_{a \in \P^d(\OC_K/ \varpi^{n+1} \OC_K)} \left( \frac{\tilde{a}_0z_0+\cdots \tilde{a}_dz_d}{z_0}\right)^{q^{n}} \pmod{ \Of^*(\bar{U}_n)^{\tilde{N}}} 
\end{equation} 
o\`u $\tilde{a}$ sont des relevés $a$ 
 dans $K^{d+1}$ par des vecteurs unimodulaires.

\end{theointro}  

\begin{rem}
Fixons une fonction $u$ vérifiant les congruences du théorème principal.
\begin{itemize}
\item Nous utiliserons de mani\`ere cruciale les r\'esultats de \cite{J1} \`a savoir l'annulation du groupe de Picard de $\H_C^d$ et l'isomorphisme entre sections inversibles et repr\'esentation de Steinberg. La suite exacte de Kummer entraine en particulier la surjectivité de la flèche  $\kappa$.  Par calcul explicite des fonctions inversible sur $\H_C^d$, on obtient \[\het{1}(\H_C^d, \mu_N)^{\gln_{d+1}(\OC_K)}= \langle \kappa(u^{q-1}) \rangle_{\Z/N\Z\modut}\] et on utilise la description locale du premier revêtement montrée dans \cite[Lemme 2.3.7.]{wa} pour voir que $\Sigma^1_C$ correspond bien au générateur $\kappa(u^{q-1})$. 
\item On voit que $\Sigma^1_C$ est un  revêtement de type Kummer pour une fonction inversible de la forme $u^{q-1}$ et nous nous en servons pour donner une nouvelle preuve du résultat classique qui affirme que le premier revêtement a $q-1$ composantes connexes géométriques (cf \ref{proppio0} et \ref{coropi0})

\item Pour simplifier l'énoncé du théorème précédent, nous avons étendu les scalaires à $C$ mais $\Sigma^1/\breve{K}$ est encore un revêtement type Kummer associé à une fonction que nous pouvons calculer.

\item Par construction, la fonction $u$ introduite dans l'énoncé du théorème est invariante sous l'action de $\gln_{d+1}(\OC_K)$ dans $\Of^{*}(\H_C^d)/\Of^{*}(\H_C^d)^N$. Cette condition est nécessaire à l'existence d'une action de $\gln_{d+1}(\OC_K)$ sur $\Sigma^1_C$ qui commute au revêtement. Mais, une telle action n'est pas unique et chacune se déduit  de l'autre en tordant par un caractère de $\gln_{d+1}(\OC_K)$ (cf \ref{propgequiv} et \ref{remequiv}). Un travail supplémentaire doit être effectué pour déterminer le caractère associé à l'action provenant de l'interprétation modulaire de $\Sigma^1_C$. 
\end{itemize}

\end{rem}

Nous terminerons par deux applications. Dans un premier temps, nous calculerons les sections inversibles du premier rev\^etement. Pour \'enoncer le r\'esultat, reprenons la description de $\Sigma^1_C$ comme revêtement de type Kummer et donnons-nous $u\in \Of^*(\H^d_C)$ vérifiant les congruences \eqref{eq:ucong} ainsi que $t \in \Of^*(\Sigma^1_C)$ tel que $t^N=  u^{q-1}$. On note $t_0:= \frac{t^{\tilde{N}}}{u}$ et on d\'efinit $\Lf_a$, pour tout $a$ dans $\mu_{q-1}(C)$,  le polyn\^ome interpolateur de Lagrange associ\'e \`a $a$ i.e. l'unique polyn\^ome unitaire de degr\'e $q-2$ qui vaut $1$ en $a$ et $0$ sur $\mu_{q-1}(C) \setminus \{a \}$. 

\begin{theointro}\label{theointrosigst} 
Toute fonction inversible $v$ de $\Sigma^1_C$ admet une unique \'ecriture \[v= \sum_{a\in \mu_{q-1}(C)} t^{j_a}v_a \Lf_a (t_0)\] o\`u $(v_a)_a$ sont des sections inversibles de $\H^d_{C}$ et $0\le j_a\le  \tilde{N}-1$. 
\end{theointro} 

\begin{rem}

\begin{itemize}
\item Les sections inversibles de $\H^d_{C}$ ont été calculées dans \cite[Théorème 2.1.]{J2} ou \cite[Théorème 7.1.]{J1} et le théorème précédent décrit bien explicitement les sections inversibles de $\Sigma^1_C$.
\item Les fonctions $\Lf_a (t_0)$ introduites précédemment sont en fait les idempotents associés aux composantes connexes de $\Sigma^1_C$. Le théorème précédent affirme alors que, sur chaque composante, les fonctions inversibles sur le premier revêtement sont engendrées par $t$ et les fonctions de la base $\H^d_C$.
\end{itemize}

\end{rem}

Pour prouver le r\'esultat, on étudie $\nrm: \Of^*(\Sigma^1_C) \to \Of^*(\H_C^d)$ la norme galoisienne du revêtement et on se ram\`ene \`a prouver que $\ker(\nrm)= (\mu_{N}(C))^{\pi_0(\Sigma^1)}$. Pour cela, on utilise un analogue du principe du maximum pour \'etudier les sections inversibles sur chaque sommet de l'immeuble de Bruhat-Tits. Ici le r\'esultat peut \^etre \'etabli par calcul direct car on a un mod\`ele lisse dont la fibre sp\'eciale est explicite.

Enfin, nous donnons aussi une description précise de $\XG[\Pi_D]$. 
L'\'enonc\'e du th\'eor\`eme principal \ref{theointroeq} d\'ecrit un ouvert de la fibre g\'en\'erique et les articles \cite{teit2} et \cite[sections 2.3, 2.4]{wa} donnent la fibre sp\'eciale gr\^ace \`a la classification des sch\'emas de Raynaud. Ces deux donn\'ees caract\'erisent enti\`erement $\XG[\Pi_D]$. En particulier, on retrouve le Théorème \ref{theointroxsig}  et le Corollaire \ref{corointrofond}  d\'ej\`a pr\'esents dans \cite{pan} quand $d=1$, et $K=\Q_p$.

Rappelons quelques notations avant d'énoncer le résultat. Le modèle entier $\H_{\OC_K}^d$ de l'espace symétrique admet un recouvrement par de affines  formels dont les ouverts $(\H_{\OC_K,  \sigma}^d)_{\sigma}=(\spf(\hat{A}_{\sigma}))_{\sigma}$ sont indexés par les simplexes maximaux ${\sigma}$ de l'immeuble de Bruhat-Tits $\BC\TC$ associé à $\pgln_{d+1}(K)$. Ces derniers sont permutés transitivement par l'action de $\pgln_{d+1}(K)$ et leurs sections $\hat{A}_{\sigma}$ est la complétion $p$-adique d'un anneau de la forme\footnote{cf \ref{sssectionhdkcalcexp} pour l'expression de la fonction $P_{\sigma}$.} $\OC_K[x_0, \dots, x_d, \frac{1}{P_{\sigma}}]/(x_0 \dots x_d- \varpi)$. Par équivariance du problème modulaire, il suffit d'étudier $\XG[\Pi_D]_{\sigma}$ est la restriction de $\XG[\Pi_D]$ à  $\H_{\OC_K,  \sigma}^d$. Le résultat est le suivant.

\begin{theointro}\label{theointroxsig} 
Soit $\sigma$ un simplexe maximal, on a\[\XG[\Pi_D]_{\sigma} \cong \spf(\hat{A}_{\sigma} [z_0, \dots, z_d]/ \langle z_i^q-u_ix_{i+1}z_{i+1} \rangle_{i \in \Z/(d+1)\Z} )\]
où 
\[u_i=\left( \frac{\prod_{(a_j) \in \F^d} 1+ \tilde{a}_1 x_{i-1}+ \dots + \tilde{a}_d x_{i-1} \dots x_{i+d-2}}{ \prod_{(a_j) \in \F^d} 1+ \tilde{a}_1 x_{i}+ \dots + \tilde{a}_d x_{i} \dots x_{i+d-1}} \right) \in \hat{A}_{\sigma}^*.\]

\end{theointro} 

  Nous terminons cette discussion en énonçant une conséquence remarquable au théorème précédent. Pour cela, on rappelle que l'idéal d'augmentation de $\XG[\Pi_D]$ est muni d'une action de $(\OC_D/\Pi_D)^*=\F_{q^{d+1}}^*$ dont les parties isotypiques sont des fibrés en droite sur $\H_{\OC_K}^d$. On appelle $\Lf_0, \dots, \Lf_d$ les parties isotypiques associées aux caractères $\F_{q^{d+1}}^*\to \bar{\F}^*_q$ obtenu en restreignant les $d+1$ plongements  $\F_q$-linéaires  de $\F_{q^{d+1}}$ dans $\bar{\F}_q$. Nous renvoyons à la classification de Raynaud \cite{Rayn} ou à \ref{deficar} pour voir le rôle-clé joué par les faisceaux $(\Lf_i)_i$ dans la géométrie de $\XG[\Pi_D]$ et de $\Sigma^1$. Le résultat remarquable est le suivant.

\begin{corointro}\label{corointrofond} 
Le faisceau $\Lf_0 \otimes \dots \otimes \Lf_d$ est trivial sur $\H_{\OC_K}^d$. 
\end{corointro}  

Ce r\'esultat permet en particulier de construire un mod\`ele semi-stable en dimension $1$ (voir \cite[section 8]{pan}). 

Expliquons le lien entre les deux résultats précédents. La description explicite du Théorème \ref{theointroxsig} permet de construire sur chaque simplexe des générateurs locaux de $\Lf_0 \otimes \dots \otimes \Lf_d$. Il s'agit de voir que ces derniers se recollent pour définir un générateur global. Mais il vérifie tous la même équation fonctionnelle dont les solutions  sont reliées aux idempotents en fibre générique via les polynômes interpolateurs de Lagrange (voir la discussion précédent \ref{theointrosigst}) et cela garantit que le procédé se globalise. Grâce à ces méthodes, on peut construire une bijection entre une famille de générateurs $(r_i)_i$ de $((\Lf_0 \otimes \dots \otimes \Lf_d)^{\otimes i})_{1\le i \le q-1}$ et les idempotents associés aux composantes connexes géométriques de la fibre générique. Le résultat précédent peut alors être vu comme une manifestation au niveau entier du faite que $\Sigma^1$ a $q-1$ composantes connexes géométriques.

\subsection*{Conventions\label{paragraphconv}}
Dans tout l'article, on fixe $p$ un premier. Soit $K$ une extension finie de $\Q_p$ fixée, $\mathcal{O}_K$ son anneau des entiers, $\varpi$ une uniformisante et $\F=\F_q$ son corps r\'esiduel. On note $C=\hat{\bar{K}}$ une complétion d'une clôture algébrique de $K$ et $\breve{K}$ une complétion de l'extension maximale non ramifiée de $K$. Soit $L\subset C$ une extension complète de $K$ susceptible de varier, d'anneau des entiers $\mathcal{O}_L$, d'idéal maximal  $\mG_L$ et de corps r\'esiduel $\kappa$. $L$ pourra être par exemple spécialisé en $K$, $\breve{K}$ ou $C$.

Soit $D$ l'alg\`ebre centrale simple sur $K$ d'invariant $\frac{1}{d+1}$ et $\OC_D$ l'anneau des entiers de $D$. Si $K_{(d+1)}$ est une extension non ramifi\'ee de $K$ de degr\'e $(d+1)$ contenue dans $D$ et $\OC_{(d+1)}$ l'ordre maximal associ\'e, alors il existe $\Pi_D$ dans $\OC_D$ tel que $\OC_D= \OC_{(d+1)} \{ \Pi_D \}$ avec $\Pi_D^{d+1}= \varpi$ et pour tout $a$ dans $\OC_{(d+1)}$, $\Pi_D \cdot a = \sigma(a) \cdot \Pi_D$. Si $R$ est une $\OC_K$-alg\`ebre, on note $R[ \Pi ]$ l'alg\`ebre $R[X]/(X^{d+1}- \varpi)$.

Si $X$ est un $L$-espace analytique réduit, 
$\Of^+_X$ sera le sous-faisceau des sections à  puissances bornées, $\Of^{++}_X$ le sous-faisceau des sections topologiquement nilpotentes. On note $\Of^*_X$ ou $\G_{m,X}$ le faisceau des sections inversibles de $\Of_X$, $\Of^{**}_{X}$ le sous-faisceau $1+ \Of^{++}_X$.


\subsection*{Remerciements}

Le présent travail a été, avec \cite{J1,J2,J4}, en grande partie réalisé durant ma thèse à l'ENS de lyon,  et a pu bénéficier de nombreux conseils et de l'accompagnement constant de mes deux maîtres de thèse Vincent Pilloni et Gabriel Dospinescu. Je les en remercie très chaleureusement. Ce travail s'inspire aussi de discussions très fructueuses avec Arnaud Vanhaecke (et indirectement Alain Genestier) qui m'ont permises de déterminer complètement la fonction caractérisant le revêtement de type Kummer. Je lui suis aussi reconnaissant pour ses éclairages sur les schéma de Raynaud dont l'influence se ressent tout au long de la partie \ref{secmax}.

\section{Espace symétrique de Drinfeld\label{ssectionhdk}}

 Nous rappelons quelques résultats standards concernant la géométrie de l'espace symétrique de Drinfeld 
 en renvoyant à  (\cite[section 1]{boca}, \cite[sous-sections I.1. et II.6.]{ds}, \cite[sous-section 3.1.]{dat1}, \cite[sous-sections 2.1. et 2.2]{wa}) pour les détails.  Nous ne traitons que les aspects combinatoires ici. Pour l'approche modulaire et la construction des revêtements, nous renvoyons à la section \ref{ssecdel}.
On fixe une extension finie $K$ de $\Q_p$, une uniformisante $\varpi$ de $K$ et un entier $d\geq 1$. On note $\F=\F_q$ le corps résiduel de $K$ et $G={\rm GL}_{d+1}(K)$. 

\subsection{L'immeuble de Bruhat-Tits} \label{paragraphbtgeosimpstd}

\label{paragraphbtsimp}
Notons $\BC\TC$ l'immeuble de Bruhat-Tits associé au groupe $\pgln_{d+1}(K)$. Le $0$-squelette $\BC\TC_0$ de l'immeuble est l'ensemble des réseaux de $K^{d+1}$ à  homothétie près, i.e. $\BC\TC_0$ s'identifie à  $G/K^*\gln_{d+1}(\OC_K)$. Un $(k+1)$-uplet de sommets 
$\sigma=\{s_0,\cdots, s_k\}\subset \BC\TC_{0} $ est un $k$-simplexe de $\BC\TC_k$ si et seulement si, quitte à  permuter les sommets $s_i$, on peut trouver pour tout $i$ des réseaux $M_i$   avec $s_i=[M_i]$ tels que \[\varpi M_k\subsetneq M_0\subsetneq M_1\subsetneq\cdots\subsetneq M_k.\] 


La réalisation topologique de l'immeuble sera notée $|\BC\TC|$. On confondra les simplexes avec leur réalisation topologique de telle manière que $|\BC\TC|=\uni{\sigma}{\sigma\in \BC\TC}{}$. Les différents $k$-simplexes, vus comme des compacts de la réalisation topologique, seront appelés faces. L'intérieur d'une face $\sigma$ sera noté $\mathring{\sigma}=\sigma \backslash \uni{\sigma'}{\sigma'\subsetneq\sigma}{}$ et sera appelé cellule.


 Fixons un simplexe pointé $\sigma\in\widehat{\BC\TC}_{k}$ et considérons une présentation : \[M_{-1}=\varpi M_k\subsetneq M_0\subsetneq M_1\subsetneq\cdots\subsetneq M_k.\] En posant 
 $$\bar{M}_i=M_i/ M_{-1},$$
 on obtient un drapeau $0\subsetneq\bar{M}_0\subsetneq \bar{M}_1\subsetneq\cdots\subsetneq \bar{M}_k\cong \F_q^{d+1}$.  On note 
 $d_i={\rm dim}_{\F_q} (\bar{M}_i)-1$ et $e_i=d_{i+1}-d_{i}.$
  Nous dirons que le simplexe $\sigma$ est de type $(e_0, e_1,\cdots, e_k)$.
  
  Considérons une base $(\bar{f}_0,\cdots, \bar{f}_d)$ adaptée  au drapeau i.e. telle que $\bar{M}_i=\left\langle \bar{f}_0,\cdots , \bar{f}_{d_i}\right\rangle$ pour tout $i$. Pour tout choix de relevés $(f_0,\cdots ,f_d)$ de $(\bar{f}_0,\cdots ,\bar{f}_d)$ dans $M_0$, on a
  $$M_i=\left\langle f_0,\cdots, f_{d_i }, \varpi f_{d_i+1},\cdots, \varpi f_d \right\rangle=N_0\oplus \cdots \oplus N_i\oplus \varpi (N_{i+1} \oplus \cdots\oplus N_k),$$
  où 
  $$N_i=\left\langle f_{d_{i-1}+1},\cdots, f_{d_{i}} \right\rangle$$ avec $d_{-1}=-1$

   Si $(f_0,\cdots ,f_d)$ est la base canonique de $K^{d+1}$, nous dirons que $ \sigma $ est le simplexe standard de type $(e_0, e_1,\cdots, e_k)$. 






   Considérons la projection naturelle $M_i\setminus M_{i-1}\subset M_i\to (M_i/ \varpi M_i)/\F^*$. On choisit un sous-ensemble $R_i$ de $M_i\setminus M_{i-1}$ qui intersecte chaque fibre de la projection en un point. On fait de même avec $N_i\setminus\varpi N_i$ et la projection $N_i\to (N_i/ \varpi N_i)/\F^*$, obtenant ainsi un sous-ensemble $\tilde{R}_i$ de $N_i\setminus\varpi N_i$.

\subsection{L'espace des hyperplans $K$-rationnels} 

   On note $\HC$ l'ensemble des hyperplans $K$-rationnels dans $\P^d$. 
Si $a=(a_0,\dots, a_d)\in C^{d+1}\backslash \{0\}$, $l_a$ désignera l'application \[b=(b_0,\dots, b_d)\in C^{d+1} \mapsto \left\langle a,b\right\rangle := \som{a_i b_i}{0\leq i\leq d}{}.\] Ainsi $\HC$ s'identifie à  $\{\ker (l_a),\; a\in K^{d+1}\backslash \{0\} \}$ et à  $\P^d (K)$.
   
   Le vecteur $a=(a_i)_{i}\in C^{d+1}$ est dit unimodulaire si $|a|_{\infty}(:=\max (|a_i|))=1$. L'application 
   $a\mapsto H_a:=\ker (l_a)$ induit une bijection entre le quotient de l'ensemble des vecteurs unimodulaires 
   $a\in K^{d+1}$ par l'action évidente de $\OC_K^*$ et l'ensemble $\mathcal{H}$. 
   
   Pour $a\in K^{d+1}$ unimodulaire et $n\geq 1$, 
   on considère l'application $l_a^{(n)} $ \[b\in (\OC_{C}/\varpi^n)^{d+1}\mapsto \left\langle a,b\right\rangle\in \OC_{C}/\varpi^n\] et on note $$\HC_n =\{\ker (l_a^{(n)}),\; a\in K^{d+1}\backslash \{0\} \; {\rm unimodulaire} \}\simeq\P^d(\OC_K/\varpi^n).$$ Alors $\HC = \varprojlim_n \HC_n$ et chaque 
   $\HC_n$ est fini.


\subsection{Géométrie de l'espace symétrique}
\label{paragraphhdkrectein}

Nous allons maintenant décrire l'espace symétrique de Drinfeld $\H_K^d$. Il s'agit de l'espace  analytique sur $K$ dont les $C$-points sont \[\H_K^d (C)=\P^d (C)\backslash \uni{H}{H\in \HC}{}.\]

On a deux recouvrements admissibles croissants par des ouverts admissibles 
$\H^d_K =\uni{\mathring{U}_n}{n> 0}{}=\uni{\bar{U}_n}{n \ge 0}{}$
où 
$$\mathring{U}_n= \{z\in\P^d_{  K} : |\varpi|^{-n} > |\frac{l_a(z)}{l_b(z)}|>|\varpi|^n,\forall a\neq b \in \HC_{n}\}$$
et 
$$ \bar{U}_n= \{z\in\P^d_{  K} : |\varpi|^{-n} \ge |\frac{l_a(z)}{l_b(z)}|\ge |\varpi|^n,\forall a\neq b \in \HC_{n+1}\}.
$$

Le recouvrement par les affinoïdes $\bar{U}_n$ est Stein (cf ). Le deuxième est cofinal et chaque $ \mathring{U}_n$ est Stein.  
On note $\bar{U}_{n,L}=\bar{U}_{n}\hat{\otimes}_K L$, $\mathring{U}_{n,L}=\mathring{U}_{n} \hat{\otimes}_K L$ et $\H^d_L=\H^d_K\hat{\otimes}_K L$. 

   Il nous sera très utile pour la suite de considérer d'autres recouvrements de $\H^d_K$, reliés à l'application de réduction vers l'immeuble de Bruhat-Tits. 
  On a une application $G$-équivariante \[\tau : \H^d_K(C)\to \{\text{normes sur } K^{d+1}\}/\{\text{homothéties}\} \] donnée par $$\tau (z): v\mapsto |\som{z_i v_i}{i=0}{d}|$$ si $z=[z_0,\cdots, z_d] \in \H_K^d (C)$. L'image $\tau(z)$ ne dépend pas du représentant de $z$ car les normes sont vues à  homothétie près. Le fait de prendre le complémentaire des hyperplans $K$-rationnels assure que $\tau(z)$ est séparé et donc  une norme sur $K^{d+1}$.

D’après un résultat classique de Iwahori-Goldmann \cite{iwgo} l'espace des normes sur $K^{d+1}$ à  homothétie près s'identifie bijectivement (et de manière $G$-équivariante) à  l'espace topologique $|\BC\TC|$, ce qui permet de voir $\tau$ comme une application 
$$\tau: \H_K^d(C)\to |\BC\TC|.$$ Plus précisément, pour toute norme  $|\cdot|$ sur $K^{d+1}$, on peut trouver une base $(f_i)_{0\le i\le d}$ de $K^{d+1}$ telle que
$|\sum_i a_i f_i|=\sup_i |a_i||f_i|$ pour tous $a_i$, et 
$ |\varpi|\le |f_{i+1}|\le |f_i|\le 1$ pour $i<d$.
Les boules $B(0,r)$ centrées en $0$ et de rayon $r$ forment une suite croissante de réseaux et on a $B(0,|\varpi|r)=\varpi B(0,r)$. Se donner les rayons minimaux de ces réseaux entre $|\varpi|$ et $1$ revient à se donner une famille de nombres $(t_i)_{0\le i\le k}\in ]0,1[^{k+1}$ dont la somme vaut $1$. L'identification d'Iwahori-Goldmann revient à associer à toute norme $|\cdot|$ le point $x$ de la cellule $\mathring{\sigma}$ de poids  $(t_i)_i$ avec $\sigma=\{B(0,r)\}_r$.


\subsection{Les recouvrements par les tubes au-dessus des faces et des cellules\label{paragraphhdksimpfer}}

Soit $\sigma\in\BC\TC_{k}$ un simplexe de type $(e_0, e_1,\cdots, e_k)$ de présentation : \[\varpi M_k\subsetneq M_0\subsetneq M_1\subsetneq\cdots\subsetneq M_k.\] Nous chercherons à  décrire $$\H_{K,\sigma}^d:=\tau^{-1} (\sigma), \,\, \H_{K,\mathring{\sigma}}^d:=\tau^{-1} (\mathring{\sigma})$$ en suivant \cite[Section 7.]{ds}. 

Comme dans \ref{paragraphbtgeosimpstd}, donnons-nous une base adaptée  $(f_0,\cdots, f_d)$ et les objets $N_i$, $R_i$ et $\tilde{R}_i$ qui s'en déduisent. Tous les vecteurs de $K^{d+1}$ seront écrits dans cette base. Si $x=\sum_{i} x_i f_i$ et $y=\sum_{i} y_i f_i$ on note $\langle x,y\rangle=\sum_{i} x_i y_i$.
En particulier,  $z_j=\langle z, f_{j}\rangle$. L'espace recherché $\H_{K,\sigma}^d(C)$ est l'ensemble des points $z\in\P^d_{  K}(C)$ pour lesquels l'ensemble des boules fermées de $\tau(z)$ sont contenues dans $\sigma$. Justifions l'égalité 
$$\H_{K,\sigma}^d=\{z\in \P^d_{ K} | \, \forall a\in M_i\backslash M_{i-1}, 1=|\frac{\left\langle z,a \right\rangle}{z_{d_{i}}}|\,\text{et}\,\, |\varpi| |z_{d_{k}}|=|\varpi| |z_{d}|\le|z_{d_{i}}|\le|z_{d_{1}}|\leq \cdots\leq |z_{d}|\}.$$
Par construction, les éléments $f_{d_{i}}$ sont dans $M_i\backslash M_{i-1}$. Le premier jeu d'égalités garantit que tous les éléments de $M_i$  ne rencontrant pas $M_{i-1}$ ont même norme pour $\tau (z)$. La chaîne d'inégalités entraîne que les rayons des boules $M_i$ sont cohérents. On en déduit l'inclusion en sens direct. Réciproquement, prenons $z$ vérifiant les inégalités ci-dessus et intéressons-nous à une boule de rayon $r$ de $\tau(z)$ pour $|\varpi||z_{d}|\le r< |z_{d}|$. Il existe un entier $i<k$ pour lequel $|z_{d_{i}}|\le r<|z_{d_{i+1}}|$. Deux cas de figure se présentent alors :
\begin{itemize}
\item Si $a$ est dans $M_i$, on a l'existence  de $j$ tel que $M_j\subset M_i$ et $a\in M_j\backslash M_{j-1}$. Les inégalités considérées entraînent alors   $r\geq|z_{d_{i}}|\geq |z_{d_{j}}|=|\left\langle z,a \right\rangle|$.
\item  Sinon, on a l'existence  de $j$ tel que $M_{i+1}\subset M_j$ et $a\in M_j\backslash M_{j-1}$. On en déduit que  $|\left\langle z,a \right\rangle|\geq|z_{d_{j}}|\geq|z_{d_{i+1}}|>r$.
\end{itemize}
  Ainsi, $M_i=B(0,r)$ ce qui prouve l'inclusion en sens inverse.

 Par ultramétrie, il suffit de vérifier ces inégalités pour un système de représentants modulo $\varpi$ et $\OC_K^*$. Ainsi:
$$\H_{K,\sigma}^d=\{z\in \P^d_{ K} |\, \forall a\in R_i,\, 
1=|\frac{\left\langle z,a \right\rangle}{z_{d_{i}}}|\,\text{et}\,\, |\varpi| |z_{d}|\le|z_{d_{0}}|\le|z_{d_{1}}|\leq \cdots\leq |z_{d}|\}.$$
Par finitude de $R_i$, l'espace ci-dessus est bien un ouvert rationnel affinoïde. %

 Soit $0\leq j< d$ un entier et soit $0\le i_0\le k$ le plus petit  entier tel que $ j< d_{i_0}$. On pose     
 $$X_j=\frac{z_j}{z_{d_{i_0}}},\,\, X_d= \varpi\frac{z_{d_{k}}}{z_{d_{0}}}=\varpi\frac{z_{d}}{z_{d_{0}}}.$$ On a $| X_j | \le 1$ avec \'egalit\'e si $j \neq d_{i_0-1}$. On obtient un système de coordonnées $(X_0,\cdots,X_d)$ qui vérifie $\pro{X_{d_{i}}}{i=0}{k}=\varpi$. Pour tout $a$ dans $M_i\backslash M_{i-1}$, les quantités $\frac{\left\langle z,a \right\rangle}{z_{d_{i}}}$ s'expriment comme des polynômes $P_a (X_0,\cdots, X_d)$ à  coefficients dans $\OC_K$ que nous décrivons plus précisément dans la section suivante. En posant $P_\sigma=\pro{\pro{P_a}{a\in R_i}{}}{i=0}{k}$, on obtient la description suivante:

\begin{prop} On a 
$$\H_{K,\sigma}^d=\{(X_0,\cdots,X_d)\in \B^{d+1}_{K}|\,  \pro{X_{d_{i}}}{i=0}{k}=\varpi\,\, \text{et}\,\, 
|P_{\sigma}(X)|=1\}.$$
\end{prop}

On en déduit que $\H_{K,\sigma}^d$ admet un modèle entier $\H_{\OC_K,\sigma}^d =\spf  (\hat{A}_\sigma)$ ou $\hat{A}_\sigma$ est le complété $p$-adique de 

\begin{equation}\label{eqhatasigma}
\OC_K [X_0,\cdots, X_d,\frac{1}{P_\sigma}]/(\pro{X_{d_{i}}}{i=0}{k}-\varpi).
\end{equation}


 On obtient alors $ \H_{K,\sigma}^d=\spg (A_\sigma)$ avec $A_\sigma=\hat{A}_\sigma[\frac{1}{\varpi}]$. De même, la fibre spéciale est donnée par $ \H_{\F,\sigma}^d=\spec (\bar{A}_\sigma)$ avec \[\bar{A}_\sigma=\hat{A}_\sigma/\varpi=\F [X_0,\cdots, X_d,\frac{1}{\bar{P}_\sigma}]/(\pro{X_{d_{i}}}{i=0}{k}). \] 
 

Des arguments identiques fournissent 

$$\H_{K,\mathring{\sigma}}^d=\{z\in \P^d_{ K} |\,  \forall a\in R_i, \ 1=|\frac{\left\langle z,a \right\rangle}{z_{d_{i}}}| \,\,\text{et}\, \, 
 |\varpi| |z_{d}|<|z_{d_{0}}|<|z_{d_{1}}|< \cdots< |z_{d}|\}$$
et l'on peut remplacer $R_i$ par $\tilde{R}_i$. 

 Considérons l'affinoïde

\[C_r= \{x=(x_1,\cdots, x_r)\in \B^r_K|\, \forall a\in \OC^{r+1}_K\backslash \varpi \OC^{r+1}_K, \  1=|\left\langle (1, x),a \right\rangle| \},\]
 la polycouronne 

\[A_k=\{y=(y_1,\cdots, y_k)\in \B^k_K|\,  1>|y_1|>\cdots>|y_k|>|\varpi|\}.\]
et les morphismes 
$$ \H_{K,\mathring{\sigma}}^d \rightarrow C_{e_i -1}, \,\, 
  [z_0,\cdots, z_d]  \mapsto  (\frac{ z_{d_{i}-e_i+1} }{ z_{d_{i}} }\cdots, \frac{ z_{d_{i}-2} }{ z_{d_{i}} }, \frac{ z_{d_{i}-1} }{ z_{d_{i}} })$$
et 
 $$\H_{K,\mathring{\sigma}}^d \rightarrow A_k,\,\, 
  [z_0,\cdots, z_d]  \mapsto  (\frac{ z_{d_{0}} }{ z_{d} },\frac{ z_{d_{1}} }{ z_{d} },\cdots, \frac{ z_{d_{k-1}} }{ z_{d} }).$$

\begin{prop}[\cite{ds} 6.4]\label{propdécompsimpouv} Les morphismes ci-dessus induisent un isomorphisme
\[ \H_{K,\mathring{\sigma}}^d \cong A_k \times \pro{C_{e_i-1}}{i=0}{k}:= A_k \times C_\sigma.\]
\end{prop}

\subsection{Calculs explicites\label{sssectionhdkcalcexp}} 


Avant de passer à la tour de revêtement, nous donnons des expressions explicites aux variables $(X_i)_i$ et aux polynômes $P_a$ qui caractérisent $\H^d_{\OC_K, \sigma}$ pour des simplexes $\sigma$ particuliers. Nous appellerons $z=[z_0,\cdots,z_d]$ la variable de l'espace projectif dans lequel sera plongé tous ces espaces.

\begin{itemize}
\item[•]D\'ecrivons le tube au-dessus d'un sommet $s$. L'action de $\gln_{d+1}(K)$ est transitive sur $\BC\TC_0$ et le sommet $s$ correspond \`a un module $M_0$ que l'on identifie \`a $\OC_K^{d+1}$. Les variables $(X_i)_i$ sont définies par les relations
\[\forall j <d, X_j=\frac{z_j}{z_d} \et X_d=\varpi.\]
On a $N_0=M_0$ et on peut prendre $R_0= \tilde{R}_0= \{ (\tilde{a}_0, \dots, \tilde{a}_d) =(\tilde{a}_{0},\dots , \tilde{a}_{j-1},1,0,\dots , 0 )\}$. Si $a=(\tilde{a}_0, \dots, \tilde{a}_d) \in R_0$, alors \[ P_{a}= \frac{\left\langle z,\tilde{a} \right\rangle}{z_d}=\tilde{a}_0X_0+ \tilde{a}_1X_1+ \dots + \tilde{a}_d \]
et les sections globales de $\H^d_{\OC_K, s}$ sont donn\'ees par la compl\'etion $\varpi$-adique de l'alg\`ebre \[ \OC_K [ X_0, \dots, X_d, \frac{1}{\prod_{\tilde{a} \in R_0} P_{a}}]/(X_d- \varpi)=\OC_K [ X_0, \dots, X_{d-1}, \frac{1}{\prod_{a \in R_0} P_a}]. \]

\item[•] Prenons maintenant $\sigma$ un simplexe maximal, ils sont tous conjugu\'es sous l'action de $\gln_{d+1}(K)$ et correspondent \`a une cha\^ine de modules \[ \varpi M_d\subsetneq M_0\subsetneq M_1\subsetneq\cdots\subsetneq M_d. \]
 Les variables $(X_i)_i$ sont définies par les relations
\[\forall j< d, X_j=\frac{z_j}{z_{j+1}} \et X_0=\varpi\frac{z_d}{z_0}.\]
On a $N_i = \langle f_i \rangle$, $\tilde{R}_i= \{ f_i \}$ et $R_i= \{ (  \tilde{a}_0, \dots,  \tilde{a}_{i-1}, 1,\varpi\tilde{a}_{i+1}, \dots, \varpi\tilde{a}_d)\}$. Si $a\in R_i$,
\[ P_a= 1 + \tilde{a}_{i-1} X_{i-1} +\tilde{a}_{i-2}X_{i-1}X_{i-2} + \dots + \tilde{a}_{i-j} \prod_{k=1}^{j} X_{i-k}+ \dots + \tilde{a}_{i+d} \prod_{k=1}^{d} X_{i-k}. \]
Les sections globales de $\H^d_{\OC_K, \sigma}$ sont donn\'ees par la compl\'etion $\varpi$-adique de l'alg\`ebre \[ \OC_K [ X_0, \dots, X_d, \frac{1}{\prod_i \prod_{\tilde{a} \in R_i} P_{\tilde{a}}}]/(X_0\dots X_d- \varpi). \]

\item[•] Intéressons-nous au simplexe standard $\sigma$ de dimension $d-1$ de type $(1,d)$. Ils  correspondent \`a une cha\^ine de modules \[ \varpi M_1\subsetneq M_0\subsetneq M_1\] avec $\dim_{\F}(M_0/\varpi M_1)=1$. Les variables $(X_i)_i$  vérifient
\[ \forall  j< d, X_j=\frac{z_j}{z_{d}} \et  X_d=\varpi\frac{z_d}{z_0}.\] Nous ne chercherons pas à exprimer les polynômes $P_a$ dont la description n'éclairera pas l'exposition. Notons toutefois que ces derniers contiennent les monômes $X_j$ pour $0<j<d$ qui définissent alors des fonctions inversibles sur $\H^d_{\OC_K, \sigma}$. On a aussi la relation $X_0 X_d=\varpi$.

\item[•] En particulier, en dimension $1$, si $s$ est un sommet et $a$ un arête, alors $\hat{A}_s$ est la compl\'etion $p$-adique de \[ \OC_K[X_0, X_1, \frac{1}{X_0-X_0^q}]/(X_1- \varpi) = \OC_K[X, \frac{1}{X-X^q}]\]
et $\hat{A}_{ a}$ est la compl\'etion $p$-adique de  \[ \OC_K[X_0, X_1, \frac{1}{1-X_0^{q-1}}, \frac{1}{1-X_1^{q-1}}]/(X_0X_1- \varpi). \]

\end{itemize}

\section{La tour de revêtement}

\subsection{Modèle de Deligne\label{ssecdel}}

Donnons une interprétation modulaire aux calculs réalisés dans \ref{paragraphhdkrectein}. On appelle $\sigma$ le Frobenius sur $\bar{\F}_q$ i.e. $\sigma(x)=x^q$ et on notera encore $\sigma$ un relev\'e sur $\OC_{\breve{K}}$. Si $X$ est un sch\'ema alg\'ebrique sur $\bar{\F}_q$ ou un sch\'ema formel sur $\OC_{\breve{K}}$, on appellera ${\rm Fr} : \sigma^{-1}_* X \to X$ le morphisme de Frobenius. De m\^eme, si $R$ est une $\OC_{\breve{K}}$-alg\`ebre, on notera $R^{(\sigma)}$ l'alg\`ebre sur $\OC_{\breve{K}}$ tordue par $\sigma$.  

Soit $D$ l'alg\`ebre centrale simple sur $K$ d'invariant $\frac{1}{d+1}$ et $\OC_D$ l'anneau des entiers de $D$. Si $K_{(d+1)}$ est une extension non ramifi\'ee de $K$ de degr\'e $(d+1)$ contenue dans $D$ et $\OC_{(d+1)}$ l'ordre maximal associ\'e, alors il existe $\Pi_D$ dans $\OC_D$ tel que $\OC_D= \OC_{(d+1)} \{ \Pi_D \}$ avec $\Pi_D^{d+1}= \varpi$ et pour tout $a$ dans $\OC_{(d+1)}$, $\Pi_D \cdot a = \sigma(a) \cdot \Pi_D$. Si $R$ est une $\OC_K$-alg\`ebre, on note $R[ \Pi ]$ l'alg\`ebre $R[X]/(X^{d+1}- \varpi)$.

Pour tout simplexe $\sigma=\{s_0, \dots, s_k\}$ dans $\BC \TC_k$, on d\'efinit un foncteur 
\[ \FC_{\sigma} : \mathrm{Nilp} \to \mathrm{Ens} \]
o\`u $\mathrm{Nilp} $ est la cat\'egorie des $\OC_K$-alg\`ebres $R$ telles que $\varpi$ soit nilpotent dans $R$. Si $R$ est une telle alg\`ebre, $\FC_{\sigma}(R)$ est l'ensemble des classes d'isomorphisme des diagrammes : 
\begin{equation}
\label{diagrammeFsigma}
\xymatrix{
\varpi M_k \ar@{^{(}->}[r] \ar[d]^{\frac{\alpha_k}{\varpi}} & M_{0}  \ar@{^{(}->}[r] \ar[d]^{\alpha_0} & \dots \ar@{^{(}->}[r] & M_k \ar[d]^{\alpha_k} \\ 
L_k \ar[r]^-{\Pi} \ar@/_1cm/[rrr]^{\cdot \varpi} & L_0 \ar[r]^-{\Pi} & \dots \ar[r]^-{\Pi} & L_k \\
}
\end{equation} 
où les $L_i$ sont des faisceaux localement libres de rang $1$ sur $S:= \mathrm{Spec}(R)$, les applications $\Pi$ des morphismes de $\Of_S$-modules et les $\alpha_i$ des morphismes de $\OC_K$-modules qui v\'erifient la condition : 
\begin{equation}
\label{CNalpha}
\ker( \alpha_i(x) : M_i / \varpi M_i \to L_i \otimes_R k(x)) \subseteq M_{i-1}/ \varpi M_i 
\end{equation}  
pour tout $x\in |S|$ ($k(x)$ \'etant le corps r\'esiduel de $x$).

Soit $\sigma$ un simplexe, décrivons plus précisément le foncteur $\FC_\sigma$. Soit $R$ un objet de $ {\rm Nilp}$ et $((L_i)_i, (\alpha_i)_i, (\Pi))$ un point de $\FC_\sigma (R)$, donnons-nous une base $(f_0,\cdots ,f_d)$ adaptée au simplexe $\sigma$. Pour $0\le j\le d$ et $i$ l'entier tel que $f_j\in M_i\backslash M_{i-1}$, la condition \eqref{CNalpha} impose que $\alpha_i (f_j)$ engendre $L_i$. Quand $j\neq d_i$, on trouve un unique élément $x_j\in R$ tel que $\alpha_i (f_{d_{i}})=x_j \alpha_i (f_j)$. On note $x_{d_{i}}\in R$ l'élément vérifiant $\alpha_{i+1}(f_{d_{i}})=x_{d_i}\alpha_{i+1} (f_{d_{i+1}})$ pour $i\neq k$ et $x_{d_{i}}=x_d$, celui vérifiant $\alpha_{0}(\varpi f_{d})=x_{d}\alpha_{0} (f_{d_{0}})$ L'égalité $\Pi\circ\cdots\circ\Pi=\varpi$ entraîne  $\prod_{0 \le i \le k} x_{d_i}= \varpi$. La condition \eqref{CNalpha} implique l'inversibilité des éléments $P_a(x_0,\cdots,x_d)$ où $P_a$ sont les polynômes introduits dans \ref{paragraphhdkrectein}. Trouver un $(d+1)$-uplet $(x_0,\cdots,x_d)$ vérifiant ces deux conditions détermine les morphismes $(\alpha_i)_i$ et $(\Pi)$ et nous avons en fait démontré

\begin{prop}
$\FC_{\sigma}$ est repr\'esentable par le sch\'ema formel $\H_{\OC_K, \sigma}^d$ d\'efini pr\'ec\'edemment.
\end{prop}

Décrivons maintenant l'élément universel  $((L_i)_i, (\alpha_i)_i, (\Pi))$ sur $\Of (\H_{\OC_K, \sigma}^d)$. Introduisons $\alpha_k :M_k\to \Of (\H_{\OC_K, \sigma}^d)$ caractérisé par \[\forall 0\le  j\le d, \alpha_{k}(f_{j})=\begin{cases} 1 &\si j=d \\ X_{d_i}\cdots X_{d_{k-1}}X_j &\si j\neq d \et d_{i-1}\le j<d_i \end{cases}.\] Le diagramme suivant correspond à l'élément universel  $\Of (\H_{\OC_K, \sigma}^d)$ :

\begin{equation}
\label{diagrammeuniv}
\xymatrix{
\varpi M_k \ar@{^{(}->}[r] \ar[d]^{\frac{\alpha_k}{\varpi}} & M_{0}  \ar@{^{(}->}[r] \ar[d]^{\alpha_0} & \dots \ar@{^{(}->}[r] & M_k \ar[d]^{\alpha_k} \\ 
\Of (\H_{\OC_K, \sigma}^d) \ar[r]^-{\times X_{d_k}=\times X_{d}} \ar@/_1cm/[rrr]^{\cdot \varpi} & \Of (\H_{\OC_K, \sigma}^d) \ar[r]^-{\times X_{d_0}} & \dots \ar[r]^-{\times X_{d_{k-1}}} & \Of (\H_{\OC_K, \sigma}^d) \\
}
\end{equation} 
où $\times f$ désigne la multiplication par $f$ dans $\Of (\H_{\OC_K, \sigma}^d)$. Notons que les flèches horizontales du diagramme précédent sont injectives et cela entraîne que les flèches $\alpha_i$ sont uniquement déterminées par $\alpha_k$.

Si on consid\`ere $\sigma'= \sigma \backslash \{ s_{j} \}$ pour $j$ dans $\left\llbracket 1,k\right\rrbracket$ alors $\FC_{\sigma'}(R)$ correspond au sous-ensemble de $\FC_{\sigma}(R)$ constitu\'e des diagrammes \eqref{diagrammeFsigma} qui v\'erifient en plus la condition que $\Pi$ est un isomorphisme entre  $L_{j}$ et ${L_{j+1}}$. On obtient notamment une immersion ouverte $\H_{\OC_K, \sigma'}^d \hookrightarrow \H_{\OC_K, \sigma}^d$ et on d\'efinit $\FC$ comme la limite sur $\BC \TC$ des $\FC_{\sigma}$.    

Le foncteur $\FC$ est en fait isomorphe au foncteur, qui \`a une alg\`ebre $R$ de $\mathrm{Nilp} $ associe l'ensemble des classes d'isomorphisme des quintuplets $(\psi, \eta, T, u,r)$ avec 
\begin{itemize}[label=\textbullet]
\item $\eta$ un faisceau en $\OC_D$-modules plats, $\Z/(d+1)\Z$-gradu\'e et constructible sur $S_{\mathrm{Zar}}$ (o\`u $S:= \mathrm{Spec}(R)$),  
\item $T$ un faisceau de $\OC_S[\Pi]$-modules, $\Z/(d+1)\Z$-gradu\'e \footnote{Dans les deux points précédents, on voit $\OC_D$ (resp. $\OC_S[\Pi]$) muni de la $\Z/(d+1)\Z$-graduation telle que les éléments  $\OC_K$ (resp. $\OC_S$) sont de degré $0$ et $\Pi_D$ (resp. $\Pi$) est de degré $1$. Les actions de ces anneaux sur $\eta$ et $T$ respectent la graduation.} et tel que les composantes homog\`enes sont des faisceaux inversibles sur $S$,
\item $u : \eta \to T$ un morphisme $\OC_D$-lin\'eaire de degr\'e 0 tel que $u \otimes \Of_S : \eta \otimes_{\OC_K} \Of_S \to T$ est surjectif,  
\item $r$ un isomorphisme $K$-lin\'eaire du faisceau constant $K^{d+1}$ vers $\eta_0 \otimes_{\OC} K$,  
\end{itemize}
qui v\'erifient les conditions : 
\begin{enumerate}
\item la restriction de $\eta_i$ au lieu d'annulation $S_i$ de $\Pi : T_i \to T_{i+1}$ dans $S$ est un faisceau constant isomorphe au faisceau constant $\OC_K^{d+1}$,   
\item pour tout $x$ de $S$, $\eta_x / \Pi \eta_x \to (T_x/ \Pi T_x) \otimes k(x)$ est injective, 
\item pour tout $i$ dans $\left\llbracket 0,d\right\rrbracket$, $\bigwedge^{d+1} (\eta_i)|_{S_i} = \varpi^{-i} \bigwedge^{d+1} (\Pi^i r\OC_K^d)|_{S_i}$.  
\end{enumerate}

\begin{theo}
Le foncteur $\FC= \varinjlim \FC_{\sigma}$ est repr\'esentable par le sch\'ema formel $\H_{\OC_K}^d$.
\end{theo}


\subsection{Interprétation modulaire de l'espace de Drinfeld }

   Pour construire le premier revêtement de l'espace de Drinfeld, nous avons besoin d'une autre interprétation modulaire, ce qui demande quelques notions et notations. 

 Si $A$ est une $\OC_K$-algèbre, un \emph{$\OC_D$-module formel} sur ${\rm Spec}(A)$ (ou, plus simplement, sur $A$) est un groupe formel $X$ sur $A$ muni d'une action de $\OC_D$, notée $\iota : \OC_D \to \mathrm{End}(X)$, qui est compatible avec l'action naturelle de $\OC_K$ sur l'espace tangent $\mathrm{Lie}(F)$, \emph{i.e.} pour $a$ dans $\OC_K$, $d\iota (a)$ est la multiplication par $a$ dans $\mathrm{Lie}(X)$. Le 
  $\OC_D$-module formel $F$ est dit \emph{spécial} si $\mathrm{Lie}(X)$ est un 
  $\OC_{(d+1)}\otimes_{\OC_K} A$-module localement libre de rang $1$. On a le r\'esultat classique suivant:




\begin{prop}
  Sur un corps alg\'ebriquement clos de caract\'eristique $p$, il existe, à isogénie près, un unique 
   $\OC_D$-module formel sp\'ecial de dimension $d+1$ et de $(\OC_K$-)hauteur $(d+1)^2$. 
\end{prop}

  On notera $\Phi_{\bar{\F}_q}$ l'unique (à isogénie près) 
$\OC_D$-module formel sp\'ecial $\Phi_{\bar{\F}_q}$ sur $\bar{\F}_q$ de dimension $d+1$ et hauteur $(d+1)^2$ (l'entier $d$ étant fixé par la suite, nous ne le faisons pas apparaître dans la notation $\Phi_{\bar{\F}_q}$).

Consid\'erons le foncteur $\GC^{Dr} : \mathrm{Nilp} \to \mathrm{Ens}$ envoyant $A\in \mathrm{Nilp}$ sur l'ensemble des classes d'isomorphisme de triplets $(\psi, X, \rho)$ avec : 
\begin{itemize}[label= \textbullet] 
\item $\psi : \bar{\F}_q \to A/ \varpi A$ est un $\F_q$-morphisme, 
\item $X$ est un $\OC_D$-module formel sp\'ecial de dimension $d+1$ et de hauteur $(d+1)^2$ sur $A$, 
\item $\rho : \Phi_{\bar{\F}_q} \otimes_{\bar{\F}_q, \psi} A/ \varpi A \to X_{A/ \varpi A}$ est une quasi-isog\'enie de hauteur z\'ero. 
\end{itemize}
   
 Le théorème fondamental suivant, à la base de toute la théorie, est dû à Drinfeld : 

\begin{theo}[\cite{dr2}]\label{Drrep}
Le foncteur $\GC^{Dr}$ est repr\'esentable par $\H_{\OC_{\breve{K}}}^d$.
\end{theo} 

Appelons $\XG$ le module universel sur $\H_{\OC_{\breve{K}}}^d$. 

\begin{rem}\label{remlierep}
L'isomorphisme pr\'ec\'edent identifie les parties isotypiques de $\mathrm{Lie}(\mathfrak{X})$ et les faisceaux $T_i$.
\end{rem}

\begin{rem}
On d\'efinit le foncteur $\tilde{\GC}^{Dr}$ de la m\^eme manière que $\GC^{Dr}$ mais en ne fixant plus la hauteur de la quasi-isog\'enie $\rho$. Alors $\tilde{\GC}^{Dr}$ est, lui aussi, repr\'esentable par un schéma formel $\widehat{\MC}^0_{Dr}$ sur $ \spf (\OC_K)$, qui se  d\'ecompose $$\tilde{\GC}^{Dr}= \coprod_{h \in \Z} \GC^{Dr,(h)},$$ o\`u $\GC^{Dr,(h)}$ est d\'efini comme pr\'ec\'edemment en imposant que la quasi-isog\'enie $\rho$ soit de hauteur $(d+1)h$. Chacun des $\GC^{Dr,(h)}$ est alors isomorphe (non canoniquement) au foncteur $\GC^{Dr}$, ce qui induit un isomorphisme non-canonique $$\widehat{\MC}^0_{Dr}\cong \H^d_{ \OC_{\breve{K}}}\times \Z.$$ 
     
\end{rem}

\subsection{La tour de Drinfeld}

 On a $\mathfrak{X}$  le $\OC_D$-module formel sp\'ecial universel sur $\H_{\OC_{\breve{K}}}^d$ (cf. th. \ref{Drrep}) et 
on note $\tilde{\XG}$ le module formel spécial universel déduit de la représentabilité de $\tilde{\GC}^{Dr}$. Pour tout entier $n\geq 1$, l'action de $\Pi_D^n$ induit une isogénie de 
 $\XG$ et de $\tilde{\XG}$. Le schéma en groupes 
$\mathfrak{X}[ \Pi_D^n] = \ker( \mathfrak{X} \xrightarrow{\Pi_D^n} \mathfrak{X})$ (resp. $\tilde{\XG}[\Pi_D^n]$)
est fini plat, de rang $q^{n(d+1)}$ sur $\H_{\OC_{\breve{K}}}^d$ (resp. $\widehat{\MC}^0_{Dr}$). 

  On note $\Sigma^0=\H^d_{\breve{K}}$ et $\MC^0_{Dr}=(\widehat{\MC}^0_{Dr})^{\rm rig}\cong\H^d_{\breve{K}}\times \Z$. Pour $n\geq 1$ on 
 d\'efinit 
\[ \Sigma^n := (\mathfrak{X}[\Pi_D^n] \backslash \mathfrak{X}[ \Pi_D^{n-1} ])^{\text{rig}},\  \MC^n_{Dr}:= (\tilde{\XG}[\Pi^n_D]\backslash\tilde{\XG}[\Pi_D^{n-1}])^{\rm rig}.\]
Les  morphismes d'oubli $\Sigma^n \to \Sigma^0$ et $\MC^n_{Dr}\to\MC^0_{Dr}$ d\'efinissent des rev\^etements finis \'etales de groupe de Galois\footnote{De même, les morphismes intermédiaires $\Sigma^n\to\Sigma^{n-1}$ et $\MC^n_{Dr}\to\MC^{n-1}_{Dr}$ sont des revêtements finis étales de groupes de Galois $(1+\Pi^{n-1}_D \OC_D)/(1+\Pi^n_D \OC_D)$. Les tours obtenues définissent aussi des revêtements pro-étales de groupe de Galois $\OC_D^*$.} $\OC_D^{*}/(1+ \Pi^n_D \OC_D)$. On a encore des isomorphismes non-canoniques $\MC^n_{Dr}\cong \Sigma^n\times \Z$ et les revêtements respectent ces décompositions.

  Le groupe $G=\gln_{d+1}(\OC_K)$ s'identifie au groupe des quasi-isogénies de $\XG$, il agit donc naturellement sur chaque niveau de la tour $(\MC^n_{Dr})_{n\geq 0}$. De même, le groupe $\OC_D^{*}$ permute les points de $\Pi^n_D $-torsion et $\OC_D^{*}$ agit sur $\MC^n_{Dr}$ à travers son quotient $\OC_D^{*}/(1+ \Pi^n_D \OC_D)\simeq {\rm Gal}(\MC^n_{Dr}/\MC^0_{Dr})$. Ces deux actions commutent entre elle et les revêtements $\MC^n_{Dr}\to\MC^0_{Dr}$ sont $G$-équivariants. En revanche, le revêtement $\Sigma^n\to\Sigma^{0}$ est seulement  $\gln_{d+1}(\OC_K)$-équivariant\footnote{En fait, il l'est même pour le groupe  $(v\circ \det)^{-1}((d+1)\Z)\subset G$.}.

\subsection{Le premier revêtement\label{paragraphrevsigman}} 

Nous nous intéressons désormais au cas $n=1$. L'espace $\Sigma^1$ est un rev\^etement  de $\Sigma^0$, de groupe de Galois  $\F_{q^{d+1}}^*$. Ce groupe est cyclique et son  cardinal,
 $$N=q^{d+1}-1,$$ est premier \`a $p$. C'est un revêtement modérément ramifié et ces deux propriétés joueront un rôle central dans l'étude que nous voulons mener. 
  Le schéma 
 $\mathfrak{X}[ \Pi_D]$ est, en particulier, un sch\'ema en $\F_p$-espaces vectoriels et la condition que $\mathfrak{X}$ soit sp\'ecial entraîne que $\mathfrak{X}[ \Pi_D]$ est un sch\'ema de Raynaud. En utilisant la classification de ces sch\'emas, il a \'et\'e montr\'e, par Teitelbaum \cite[ théorème 5]{teit2} pour $d=1$ et par Wang  \cite[ lemme 2.3.7.]{wa} pour $d$ quelconque: 

\begin{theo}  \label{theoniventiersigmasommet}
Soit $s$ le sommet standard de $\BC\TC$ et $$u_1=\prod_{H \in \HC_1 \backslash \{ H_0 \}} (\frac{l_H}{l_{H_0}})^{q-1},$$ avec $H_0\in\HC_1$ une direction \`a l'infini privil\'egi\'ee et $l_H=:l_{a_H}$ une forme linéaire pour $a_H\in K^{d+1}$ unimodulaire tel que $H=\ker(l_{a_H}^{(1)})$. On a
 \[ \XG[\Pi_D]  |_{\H_{\OC_{\breve{K}},s}^d}=\spf (\hat{A}_s[T]/(T^{q^{d+1}}-\varpi u_1 T))  \]
avec $\hat{A}_s=\Of(\H_{\OC_{\breve{K}},s}^d)$. \end{theo}

   Nous allons rappeler les arguments dans la section suivante pour le confort du lecteur. En effet, le théorème précédent décrit le « niveau entier » du premier revêtement alors que le résultat de Wang décrit la fibre générique. Toutefois, ce résultat plus fort est utilisé dans la preuve même s'il n'est pas énoncé explicitement. Nous le préférons car il détermine tous les objets qui nous intéressent (fibre générique et fibre spéciale).
 
 On peut encore définir une flèche de réduction de $\Sigma^1$ vers l'immeuble de Bruhat-Tits $\BC\TC$ s'inscrivant dans le diagramme :
\[
\xymatrix{ 
\Sigma^1  \ar[d]^{\pi} \ar[dr]^{\nu} & \\
 \H^d_{\breve{K}} \ar[r]^-{\tau} & \BC \TC }.\]
Pour tout sous-complexe simplicial $T \subseteq \BC \TC$, on note 
$$\Sigma^1_{T}= \nu^{-1}(T),\,\,\Sigma^1_{L,T}= \Sigma^1_{T}\otimes_{\breve{K}} L.$$ Le théorème \ref{theoniventiersigmasommet} établit l'isomorphisme $$\Sigma^1_s\cong \H^d_{\breve{K},s}((\varpi u_1)^{\frac{1}{N}}).$$ 

Nous allons aussi généraliser cette description au revêtement  $\Sigma^1/\H^d_{{\breve{K}}}$ tout entier.

\section{Description de $\XG[\Pi_D]  |_{\H_{\OC_{\breve{K}},s}^d}$}

\subsection{Schémas de Raynaud\label{ssectionrayn}}

Soit $S$ un sch\'ema formel sur $\spf (\OC_{\breve{K}})$,\footnote{Nous pouvons en fait raisonner sur un anneau des entiers beaucoup plus petit que $\OC_{\breve{K}}$ appelé $D$ dans \cite{Rayn} (il est en particulier fini sur $\Z$). Toutefois l'anneau $\OC_{\breve{K}}$ est le cadre naturel pour le problème modulaire étudié.} nous nous int\'eressons aux sch\'emas $X$ en groupes finis, plats et  de présentation  finie  sur $S$, muni d'une action de $\F_{q^{d+1}}$ avec $q=p^f$. En particulier, $X$ est affine si $S$ l'est.

  Si on note $\IC$ l'id\'eal d'augmentation de $X$, il admet une d\'ecomposition  en parties isotypiques pour l'action du groupe commutatif $\F_{q^{d+1}}^*$,
\[ \IC= \bigoplus_{\chi \in (\F_{q^{d+1}}^*)^{\vee}} \Lf_{\chi}. \]

\begin{defi}
On dit que $X/S$ est un sch\'ema en $\F_{q^{d+1}}$-espaces vectoriels de Raynaud si les parties isotypiques $\Lf_{\chi}$ sont localement libres de rang $1$ sur $\Of_S$. Si, de plus, chaque $\Lf_{\chi}$ est un libre, nous dirons que $X$ est un sch\'ema de Raynaud libre. 
\end{defi}

Notons l'existence de fl\`eches pour $\chi, \chi'$ des caract\`eres :
\[ d_{\chi, \chi'} : \Lf_{\chi} \otimes \Lf_{\chi'} \to \Lf_{\chi \chi'}, \]
\[ c_{\chi, \chi'} : \Lf_{\chi \chi'} \to \Lf_{\chi} \otimes \Lf_{\chi'}, \]
provenant de la multiplication et de la comultiplication dans $\Of_X$.

\begin{defi}
On dit que $\chi$ est fondamental si le prolongement de $\chi$ \`a $\F_{q^{d+1}}$ (envoyant $0$ sur $0$) est un morphisme de corps. 
\end{defi} 

Si $\chi_0$ est un caract\`ere fondamental, on observe que la famille $(\chi_i)_{i \in \Z/f(d+1)\Z}=(\chi_0^{p^i})_{i \in \Z/f(d+1)\Z}$ parcourt l'ensemble des caract\`eres fondamentaux. De plus, tout caract\`ere $\chi$ admet une unique 
écriture \[ \chi = \prod_{i \in \Z/f(d+1)\Z} \chi_i^{\alpha_i}\]
avec $\alpha_i=0, \dots , p-1$. On construit grâce aux morphismes $(c_{\chi, \chi'})_{\chi, \chi'}$, $(d_{\chi, \chi'})_{\chi, \chi'}$ et à l'associativité de la multiplication et de la comultiplication des applications :
\[ c_{\chi}: \Lf_{\chi} \to \Lf_{\chi_0}^{\alpha_0}\otimes\dots\otimes \Lf_{\chi_{f(d+1)-1}}^{\alpha_{f(d+1)-1}},\]     
\[ d_{\chi}: \Lf_{\chi_0}^{\alpha_0}\otimes\dots\otimes \Lf_{\chi_{f(d+1)-1}}^{\alpha_{f(d+1)-1}} \to \Lf_{\chi}.\] La compos\'ee $d_\chi\circ c_\chi$ est donc dans $\en(\Lf_{\chi_{i+1}})= \Of(S)$. 

 Les parties isotypiques $\Lf_{\chi_i}$ pour les caract\`eres fondamentaux déterminent les autres. En effet, on a le résultat suivant : 
\begin{lem}[\cite{Rayn} proposition 1.3.1]\label{lemisot}

La composition $d_{\chi}\circ c_{\chi}$ est dans $\OC_{\breve{K}}^*\subset \Of^*(S)$ et $d_{\chi}$, $c_{\chi}$ sont donc inversibles.

\end{lem}

Il reste à comprendre les relations entre les caract\`eres fondamentaux. Pour cela, nous introduisons  
\[ c_{i}: \Lf_{\chi_{i+1}} \to \Lf_{\chi_i}^{\otimes p},\]     
\[ d_{i}: \Lf_{\chi_i}^{\otimes p} \to \Lf_{\chi_{i+1}}.\] 
Comme précédemment, la compos\'ee $d_i \circ c_i$ est  dans $\en(\Lf_{\chi_{i+1}})= \Of(S)$.

\begin{prop}[\cite{Rayn} proposition 1.3.1]
Il existe $w$ une constante de $p\OC_{\breve{K}}\subset \Of(S)$ telles que pour tout  sch\'ema de Raynaud $X/S$, 
\[ \forall i, d_i \circ c_i = w \id_{\Lf_{\chi_{i+1}}}. \] 
\end{prop} 

Nous pouvons \'enoncer la classification de Raynaud.

\begin{theo}[\cite{Rayn} théorème 1.4.1]
\label{theoclassicraynaud}
Soit $S$ un sch\'ema sur $\OC_{\breve{K}}$, l'application 
\[ X/S \mapsto (\Lf_{\chi_i}, c_i, d_i)_{i} \] 
induit une bijection des schémas de Raynaud sur $S$ à isomorphisme près vers les familles de faisceaux inversibles $(\Lf_i)_i$ sur $S$ munis de morphismes $c_{i}: \Lf_{i+1} \to \Lf_{i}^{\otimes p}$ et     
$d_{i}: \Lf_{i}^{\otimes p} \to \Lf_{i+1}$ v\'erifiant $d_i \circ c_i= w \id_{\Lf_{i+1}}$.  
\end{theo}

\begin{coro}\label{cororaynlib}
Soit $S=\spf(A)$ un sch\'ema formel sur $\OC_{\breve{K}}$ et $X/S$ un sch\'ema de Raynaud libre, on peut trouver des sections $(v_i)_{i \in \Z/(d+1)\Z}$ dans $A$ telles que\footnote{Notons que nous n'avons décrit uniquement les sections globales en tant qu'algèbre. La comultiplication peut aussi être donnée explicitement caractérisant ainsi la structure d'algèbre de Hopf. Nous ne nous en servirons pas pour la suite.} 
\[ X=\spf(A[y_0, \dots, y_{f(d+1)-1}]/(y_i^p-v_i y_{i+1})). \] 
La famille $(v_i)_i$ est unique modulo la relation d'\'equivalence $(v_i)_i \sim (v'_i)_i$ si et seulement si il existe $(u_i)_i$ dans $A^*$ telle que $v'_i=v_i \frac{u_i^p}{u_{i+1}}$. 
\end{coro}

\begin{proof}
Fixons un caract\`ere $\chi= \prod_i \chi_i^{\alpha_i}$ et choisissons $y_i \in \Of(X)$ un g\'en\'erateur de $\Lf_{\chi_i}$ pour tout caract\`ere fondamental $\chi_i$. Dans ce cas, $\Lf_{\chi}$ est engendr\'e par $\prod y_i^{\alpha_i}$ d'après \ref{lemisot} et le morphisme $d_i : \Lf_{\chi_i}^{\otimes p} \to \Lf_{\chi_{i+1}}$ induit une relation $y_i^p=v_iy_{i+1}$ avec $v_i \in A$. Comme $X=\spf(A \bigoplus_{\chi} \Lf_{\chi}(S))$, on en déduit une flèche surjective :
\[ A[y_0, \dots, y_{f(d+1)-1}]/(y_i^p-v_i y_{i+1})\to \Of(X). \] 
C'est un bijection par classification de Raynaud \ref{theoclassicraynaud}. Le choix d'un autre jeu de g\'en\'erateurs $(y_i')_i=(y_iu_i)_i$ avec $u_i \in A^*$ entraîne l'\'ecriture 
\[ \Of(X)=A[y'_0, \dots, y'_{f(d+1)-1}]/({y'_i}^p-v_i \frac{u_i^p}{u_{i+1}} y'_{i+1}). \] 
 Ainsi choisir une classe d'\'equivalence des $(v_i)_i$ d\'etermine les morphismes $d_i$ et aussi les morphismes $c_i$ via la relation $d_i \circ c_i= w \id_{\Lf_{\chi_{i+1}}}$.   Ceci entraîne l'unicité des $(v_i)_i$  par classification des schémas de Raynaud.
\end{proof}

Expliquons maintenant le lien avec les $\OC_D$-modules formels sp\'eciaux. Dor\'enavant, $X$ d\'esignera un $\OC_D$-module formel. Dans ce cas, $X[\Pi_D]$ est un sch\'ema en $\F_{q^{d+1}}$-espaces vectoriels. 
 On notera encore $\IC$, $\Lf_{\chi_i}$ l'id\'eal d'augmentation et les parties isotypiques de $X[\Pi_D]$.  

\begin{defi}\label{deficar}

On écrira $\Lambda$ le sous-ensemble des caractères fondamentaux qui sont $\F$-linéaires. Quitte à réaliser une  permutation sur les caractères fondamentaux, on peut supposer $\chi_0\in\Lambda$. Dans ce cas, on a l'égalité \[\Lambda=\{\chi_0^{q^i}:0\le i\le d\}=\{\chi_{fi}:0\le i\le d\}.\]
Pour simplifier, on pourra écrire $\Lf_i=\Lf_{\chi_{fi}}$.

\end{defi}

La proposition suivante exhibe le lien entre les faisceaux $\Lf_{\chi_i}$ et la d\'ecomposition en parties isotypiques de $\lie(X)$ (voir \ref{remlierep}). 
\begin{prop}
\label{propliexpi}
On a d'une part 
\[ \omega_{X[\Pi_D]/S}= \IC/\IC^2= \bigoplus_{j\in\Z/f(d+1)\Z} \Lf_{\chi_{j}}/ d_j(\Lf_{\chi_{j-1}}^{\otimes p}) \]
et d'autre part 
\[ \omega_{X[\Pi_D]/S}= \omega_{X/S}/\Pi_{D,*}\omega_{X/S}= \bigoplus_{i\in\Z/(d+1)\Z}(\lie (X)_{d+1-i})^\vee / \Pi_D^*(\lie (X)_{d+2-i})^\vee. \] 
Ces décompositions font intervenir les parties isotypiques de $\omega_{X[\Pi_D]/S}$, et on en déduit :
\[\forall j \in \Z/f(d+1)\Z,\Lf_{\chi_{j}}/d_{j-1}(\Lf_{\chi_{j-1}}^{\otimes p})\cong \begin{cases}(\lie (X)_{d+1-i})^\vee / \Pi(\lie (X)_{d+2-i})^\vee & \si j=fi\in f\Z/f(d+1),\\ 0 &\sinon\end{cases}.\]
\end{prop} 

\begin{proof}
L'égalité $\omega_{X[\Pi_D]/S}= \IC/\IC^2$ s'obtient par définition. La décomposition en parties isotypiques passe au quotient et devient :
\[\omega_{X[\Pi_D]/S}=\bigoplus_{j\in\Z/f(d+1)\Z} \Lf_{\chi_{j}}/\Lf_{\chi_{j}}\cap \IC^2.\]
Il reste à prouver l'égalité $\Lf_{\chi_{j}}\cap \IC^2=d_{j-1}(\Lf_{\chi_{j-1}}^{\otimes p})$. L'inclusion en  sens indirect est claire. Pour le sens direct, cela découle de \ref{lemisot} et de l'associativité de la multiplication.

On rappelle que $X[\Pi_D]= \ker (\Pi_D : X \to X)$ ce qui entraîne l'égalité $\lie (X[\Pi_D])= \ker (\Pi_D^* : \lie (X) \to \lie (X))$. En passant au dual, on obtient $\omega_{X[\Pi_D]/S}= \omega_{X/S}/\Pi_{D,*}\omega_{X/S}$. On rappelle que l'endomorphisme $\Pi_D$ agit sur $\lie (X)$ en permutant de manière circulaire les parties isotypiques fondamentales $\F$-linéaires. La décomposition en parties isotypiques s'en déduit :
\[\omega_{X/S}/\Pi_{D,*}\omega_{X/S}= \bigoplus_{i\in\Z/(d+1)\Z}(\lie (X)_{d+1-i})^\vee / \Pi_D^*(\lie (X)_{d+2-i})^\vee.\]
\end{proof}

On en d\'eduit que si $X$ est sp\'ecial, $X[\Pi_D]$ est un sch\'ema de Raynaud \cite{wa}\footnote{Notons que \cite[4.15]{vh} donne une condition suffisante pour que la réciproque soit vérifiée.}. 
\begin{coro}\label{propxpidlib}

Soit $S=\spf(A)$ un sch\'ema formel sur $\OC_{\breve{K}}$ et $X/S$ un module formel spécial tel que le schéma de Raynaud $X[\Pi_D]$ soit libre, alors
\[ X[\Pi_D]=\spf(A[y_0, \dots, y_{d}]/(y_i^q-v_i y_{i+1})). \] 
La famille $(v_i)_i$ est unique modulo la relation d'\'equivalence $(v_i)_i \sim (v'_i)_i$ si et seulement si il existe $(u_i)_i$ dans $A^*$ telle que $v'_i=v_i \frac{u_i^q}{u_{i+1}}$. 

\end{coro}

\begin{proof}

Nous savons déjà (voir \ref{cororaynlib}):
\[ X[\Pi_D]=\spf(A[\tilde{y}_0, \dots, \tilde{y}_{f(d+1)-1}]/(\tilde{y}_i^p-\tilde{v}_i \tilde{y}_{i+1})). \] 

Si $\chi_i$ n'est pas un caractère $\F$-linéaire, alors $(\omega_{X[\Pi_D]/S})_{\chi_i}=0$ d'où $\Lf_{\chi_i}=d_{i-1}(\Lf_{\chi_{i-1}}^{\otimes p})$ (d'après \ref{propliexpi}) ou, dit autrement, $\tilde{v}_{i-1}\in A^*$. 
On peut ainsi exprimer $\tilde{y}_i$ en fonction de $\tilde{y}_{i-1}$. Ce qui entraîne\footnote{Nous avons ici utilisé une astuce qui nous servira dans le reste du texte. Nous avons décidé de l'énoncer ici dans le cas le plus général pour pouvoir le citer plus tard. Donnons-nous un anneau $A$ et une algèbre sur $A$ de la forme $B=A[x_1,\cdots,x_r]/I$ avec $I$ un idéal contenant des éléments $(x_i^p-v_i x_{i+1})_{1\le i<r}$ où $v_i\in A$. Si les fonctions $v_i$ sont inversibles, on a $B=A[x_1]/I'$ avec $I'=I\cap A[x_1]$, car on peut établir  par récurrence sur $2\le s\le r$ \[x_s=\frac{1}{v_{s-1}}x_{s-1}^p=\frac{1}{v_{s-1}v_{s-2}^p\cdots v_{1}^{p^{s-1}}}x_{1}^{p^{s-1}}\] où le cas $s=2$ est donné par les hypothèses. \label{foouni} } 

\[ X[\Pi_D]=\spf(A[y_0, \dots, y_{d}]/(y_i^q-v_i y_{i+1})). \]

Pour l'unicité des sections $(v_i)_i$, il suffit d'observer que chaque variable $y_i$ est vu comme un générateur du fibré $\Lf_i$ et on peut appliquer l'argument de \ref{cororaynlib}.
\end{proof}
\subsection{Application à $\XG[\Pi_D]$}

Int\'eressons-nous maintenant \`a $\XG$ le module formel sp\'ecial universel sur $\H^d_{\OC_{\breve{K}}}$. On fixe $s$ le sommet standard et $\sigma$ l'ar\^ete standarde de type $(1,d)$ et on note $\XG[\Pi_D]_s$ et $\XG[\Pi_D]_{\sigma}$ la restriction de $\XG[\Pi_D]$ au-dessus de $s$ et de $\sigma$. Le r\'esultat pr\'ec\'edent appliqu\'e au tube au-dessus de $s$ et de $\sigma$ nous donne 

\begin{coro}\label{coropic}
Il existe une fonction $u^{(s)}\in \hat{A}^*_s$ unique modulo $(\hat{A}^*_s)^N$ et des fonctions inversibles $u_0^{(\sigma)}, u_1^{(\sigma)} \in \hat{A}^*_{\sigma}$ telles que 
\[ \XG[\Pi_D]_s= \spf(\hat{A}_s [ y_1]/(y_1^{q^{d+1}}-\varpi u^{(s)} y_1 )) \]
\[ \XG[\Pi_D]_{\sigma}= \spf(\hat{A}_{\sigma}[y_0, y_1]/(y_1^{q^d}- u_d^{(\sigma)} x_0 y_0, y_0^{q}-u_0^{(\sigma)} x_d y_1)) \]  
avec $\hat{A}_s= \Of(\H^d_{\OC_{\breve{K}},s})$ et  $\hat{A}_{\sigma}= \Of(\H^d_{\OC_{\breve{K}},\sigma})$. 
\end{coro}

\begin{proof}
 On prouve dans un premier temps l'annulation des groupes de Picard de $\H^d_{\OC_{\breve{K},s}}$ et $\H^d_{\OC_{\breve{K},\sigma}}$. Pour cela, on rappelle succinctement les arguments de Wang \cite[Lemme (2.4.6)]{wa}. D'apr\`es \cite[3.7.4]{frvdp}, il suffit de raisonner en fibre sp\'eciale. Pour le sommet, la fibre sp\'eciale est le spectre d'une $\bar{\F}$-algèbre factorielle de type fini. Pour l'ar\^ete $\sigma$, on note  $(x_i)_i$  système de coordonnées sur $\H^d_{\OC_{\breve{K},\sigma}}$ introduit dans \ref{paragraphhdkrectein}. La fibre sp\'eciale admet deux composantes irr\'eductibles $V(x_0)=V_0$ et $V(x_d)=V_d$. On a une suite exacte scind\'ee \footnote{voir aussi la preuve de  \ref{lempic} pour plus de détails(pas de risque d'argument circulaire...).} 
\[ 0 \to A/(X_0X_d) \to  A/(X_0) \times  A/(X_d) \to A/(X_0,X_d) \to 0 \]
avec $A= \bar{\F}[X_0, \dots, X_d]$. 
 Soit $\pG$ un point ferm\'e de la fibre sp\'eciale $\H^d_{\bar{\F}, \sigma}$ (resp. $V_0$, $V_d$, $V_{0,d}:=V(X_d) \cap V(X_0)$), l'anneau local en ce point est un localis\'e de $A/(X_0X_d)$ (resp. $A/X_0$, $A/X_d$, $A/(X_0,X_d)$). On a la m\^eme propri\'et\'e pour les sections globales de $\H^d_{\bar{\F}, \sigma}$, $V_i$ et $V_{0,d}$. On en d\'eduit les suites exactes scind\'ees 
 \[ 0 \to \Of^*(\H^d_{\bar{\F}, \sigma}) \to \Of^*(V_0)\times \Of^*(V_d) \to \Of^*(V_{0,d})\to 0, \]
 \[ 0 \to \Of^*_{\H^d_{\bar{\F}, \sigma}, \pG} \to \Of^*_{V_0, \pG} \times \Of^*_{V_d, \pG} \to \Of^*_{V_{0,d}, \pG} \to 0, \]   
\[ 0 \to \Of^*_{\H^d_{\bar{\F}, \sigma}} \to \iota_* \Of^*_{V_0} \times \iota_* \Of^*_{V_d} \to \iota_* \Of^*_{V_{0,d}} \to 0, \] 
o\`u $\iota_*$ d\'esigne les immersions ferm\'ees de $V(X_i)$ ou $V(X_0) \cap V(X_d)$ dans $\H^d_{\bar{\F}, \sigma}$. Notons que la deuxième entraîne directement la troisième.
   

La suite exacte longue associ\'ee induit l'exactitude de 
\[ \Of^*(V_{0,d})/( \Of^*(V_0)\times \Of^*(V_d)) \to \pic(\H^d_{\bar{\F}, \sigma}) \to \pic(V_0) \times \pic(V_d). \]
On a montr\'e que le premier quotient est trivial et $\pic(V_0)=\pic(V_d)=0$ car les deux ferm\'es sont affines de sections globales factorielles. On en d\'eduit l'annulation du groupe de Picard recherch\'ee. 
 
Ainsi les deux modules formels $\XG[\Pi_D]_s$ et $\XG[\Pi_D]_{\sigma}$ sont des sch\'emas de Raynaud libres, on a d'apr\`es \ref{cororaynlib}, 
\[ \XG[\Pi_D]_s = \spf(\hat{A}_s[y_0, \dots, y_d]/ (y_i^q- v^{(s)}_i y_{i+1})),\] 
\[ \XG[\Pi_D]_{\sigma} = \spf(\hat{A}_{\sigma}[y_0, \dots, y_d]/ (y_i^q- v^{(\sigma)}_i y_{i+1})). \] 
On a par l'identité  \ref{propliexpi} et \eqref{diagrammeuniv}  
\[\forall j \in \Z/f(d+1)\Z,\Lf_{\chi_{j}}/d_{j-1}(\Lf_{\chi_{j-1}}^{\otimes p})\cong \begin{cases} (T_{d})^{\vee}/\Pi((T_{0})^{\vee}) \cong \Of/\varpi \Of & \si j=f,\\ 0 &\sinon\end{cases}. \]
Ainsi, $v_i^{(s)}$ est inversible si et seulement si $i\neq d$ et on  $v_0^{(s)}=\varpi u_0$   avec $u_{0} \in \hat{A}_{s}^*$.  En particulier, en reprenant le raisonnement dans \eqref{foouni}, obtient l'\'ecriture  
\[ \XG[\Pi_D]_s= \spf(\hat{A}_s [ y_1]/(y_1^{q^{d+1}}-\varpi u^{(s)} y_1 )). \]
L'unicité de $u$ découle des conditions d'unicité sur les $v_i^{(s)}$ décrites dans  \ref{cororaynlib}.

De m\^eme sur $\sigma$, $v_i^{(\sigma)}$ est inversible si et seulement si $i \neq 0,d$ d'où $v_{0}^{(\sigma)}=x_d u_{0}$ et $v_d^{(\sigma)}= x_0 u_d$ avec $u_0^{(\sigma)},u_d^{(\sigma)}\in  \hat{A}_{\sigma}^*$. Ainsi, en reprenant le raisonnement dans \eqref{foouni}, on a 
l'écriture \[ \XG[\Pi_D]_{\sigma}= \spf(\hat{A}_{\sigma}[y_0, y_1]/(y_1^{q^d}- u_d^{(\sigma)} x_0 y_0, y_0^{q}-u_0^{(\sigma)} x_d y_1)). \]   

\end{proof}

\begin{coro}\label{corouni}

En reprenant les notations ci-dessus, on a la congruence :

\[u^{(s)}\equiv(u_d^{(\sigma)})^{q} u_0^{(\sigma)} x_0^{q-1} \pmod{(\hat{A}_s^{*})^N}. \]

\end{coro}

\begin{proof}

Il s'agit d'observer que $\XG[\Pi_{D}]_{s}= \spf(\Of(\XG[\Pi_{D}]_{\sigma})[1/x_0])$. Ainsi, en reprenant le raisonnement dans \eqref{foouni}, on a 
\[
\XG[\Pi_D]_s= \spf(\hat{A}_s [ y_1]/(y_1^{q^{d+1}}-\varpi (u_d^{(\sigma)})^{q} u_0^{(\sigma)} x_0^{q-1}y_1 )) 
\]
et on conclut par unicité de $u^{(s)}$. 

\end{proof}

On peut maintenant terminer la preuve de \ref{theoniventiersigmasommet} et déterminer un représentant de $u^{(s)}$. On a par lissité de $\H^d_{\OC_{\breve{K}}, s}$, \[\hat{A}^*_s = \OC_{\breve{K}}^*(1+\hat{A}^{++}_s) \prod_{H} ({\frac{l_H}{l_{H_0}}})^{\Z}\]  où $H$ parcourt un système des représentants $S_1$ des hyperplans modulo $\varpi$. Ainsi $\hat{A}^*_s/ (\hat{A}^*_s)^N=\prod_{H} ({\frac{l_H}{l_{H_0}}})^{\Z/N\Z}$ et $u^{(s)}$ peut être choisi de la forme $ \prod_{H} l_{H}^{\alpha_H}$ avec $\sum_H \alpha_H \equiv 0 \pmod{N}$. 

Mais $\XG[\Pi_D]_{s}$ admet une action de $\gln_{d+1}(\OC_K)$ équivariante sous $\F_{q^{d+1}}^*$. Cette action doit préserver  les parties isotypiques de l'idéal d'augmentation  et $g\cdot y_1$ est encore un générateur de $\Lf_1$ pour $g\in\gln_{d+1}(\OC_K)$ d'où $g\cdot y_1=\lambda_g y_1$ avec $\lambda_g \in \hat{A}^*_s$. De plus, on a la relation $(g\cdot y_1)^{q^{d+1}}=\varpi (g\cdot u^{(s)}) g\cdot y_1$. En particulier, \[y_1^{q^{d+1}}=\varpi (g\cdot u^{(s)}) \lambda_g^{-N} y_1\]
et l'unicité dans la classification de Raynaud entraîne que $u^{(s)}=(g\cdot u^{(s)}) \lambda_g^{-N}=(g\cdot u^{(s)})$ dans $ \hat{A}^*_s/ (\hat{A}^*_s)^N$ et $u^{(s)}$ est invariante  modulo $(\hat{A}^*_s)^N$ sous l'action de $\gln_{d+1}(\OC_K)$.  Comme l'action est transitive sur les hyperplans, $\alpha:= \alpha_H$ ne dépend pas de $H$ et $|S_1|\alpha \equiv  0 \pmod{N}$. Ainsi $(q-1)|\alpha$ car $|S_1|=N/(q-1)$. 

Le reste de la preuve consiste à relier ces entiers à des ordres d'annulation pour certains hyperplans d'un espace projectif bien choisi puis de calculer cet ordre pour un hyperplan particulier. Pour cela, on se place sur l'arête $\sigma$ considérée précédemment.  Le tube en niveau entier de cette arête $\H_{\OC_K, \sigma}^d=\spf(\hat{A}_\sigma)$ admet deux composantes irréductibles en fibre spéciale $V(x_0)$ et $V(x_d)$ et on a une immersion ouverte $V(x_d)\flinj\P^d_\F:=\proj(\F[z_0,\cdots, z_d])$ qui se factorise par l'ouvert $D^+(z_d)\subset \P^d_\F$ et qui est induite au niveau des sections globales par l'application 
\[\frac{z_i}{z_d}\in \Of(D^+(z_d))\mapsto x_i \in \hat{A}_\sigma/(x_d).\]
Si on compose cette flèche par le plongement naturel $\H_{\F, s}^d\flinj V(x_d)$, on obtient une autre immersion ouverte $\H_{\F, s}^d\flinj  \P^d_\F$ induite par l'application 
\[\frac{z_i}{z_d}\in \Of(D^+(z_d))\mapsto x_i \in \hat{A}_s\]
et l'image de ce morphisme s'identifie à $\P^d_\F\backslash \bigcup_{H\in \HC_1}H$. 

Pour une fonction inversible $u=\prod_{H} l_{H}^{\beta_H}$ sur $\H_{\F, s}^d$ vu comme une fonction méromorphe de $\P^d_\F$, les entiers $\beta_H$  correspondent aux ordres $\ord_H(u)$ en l'hyperplan $H\subset \P^d_\F$. Posons $H_0=V^+(z_0)$, on a la relation d'après \ref{corouni}
\[
\ord_{H_0}(u^{(s)})\equiv (q-1)\ord_{H_0}(x_0)+q\ord_{H_0}(u_d^{(\sigma)})+\ord_{H_0}(u_0^{(\sigma)})  \pmod{N}.
\]
Par construction, on a $\ord_{H_0}(x_0)=\ord_{H_0}(z_0/z_d)=1$. Notons aussi que
\[H_0\cap V(x_d) =\spec(\bar{A}_\sigma/(x_0,x_d))=\spec(\F[x_1,\cdots, x_{d-1}, 1/P_\sigma(0,x_1,\cdots, x_{d-1},0)])\neq \emptyset.\]
Ainsi, les fonctions inversibles sur $\H_{\F, \sigma}^d\supset V(x_d)$ comme $u_0^{(\sigma)}$ et $u_d^{(\sigma)}$ ont un ordre trivial en $H_0$. On en déduit que $\alpha=q-1$ ce qui conclut la preuve de \ref{theoniventiersigmasommet}.

\section{\'Equations pour le premier revêtement de l'espace symétrique de Drinfeld}

L'objectif de cette section est de prouver le théorème \ref{theointroeq} et ainsi donner une équation pour le premier revêtement $\Sigma^1$ de $\H^d_{ \breve{K}}$. Ce dernier a un groupe de Galois cyclique d'ordre $N=q^{d+1}-1$ qui est premier à $p$ et on le voit comme un $\mu_N$-torseur ou encore comme une classe du groupe $\het{1}( \H_{ \breve{K}}^d, \mu_N)$. Cette cohomologie est bien comprise grâce aux résultats d'annulation de \cite{J1} que nous rappelons dans \ref{sssectionrecaphdk}. Nous obtenons alors une classification explicite des $\mu_N$-torseurs et nous l'appliquerons dans \ref{sssectioneqsigma1} pour établir le résultat voulu.

\subsection{Quelques faits généraux sur les $\mu_n$-torseurs}

Nous décrivons dans cette section certaines classes de revêtement sur des espaces analytiques ainsi que leur interprétation cohomologique. Soit $X$ un espace analytique lisse et $N$ un entier premier à  $p$
, toute classe de $\het{1}(X, \mu_N)$ s'identifie   à un revêtement galoisien  $\pi: \TC \to X$  de groupe de Galois $\mu_N$. Le morphisme de Kummer  sera noté $\kappa : \Of^* (X)\to  \het{1}(X,\mu_N)$ et 
le torseur $\kappa(u)$ associé à  une section inversible $u$ de $X$ est défini par \[X(u^{1/N}):=\underline{\spec}_{\Of_X}  (\Of_X[T]/(T^N -u)).\]

Nous commençons par décrire le nombre de composantes connexes géométriques d'un revêtement de type Kummer.

\begin{prop}\label{proppio0} 
Soit $X=\spg(A)$ un $C$-affinoide lisse   connexe et 
$u\in A^*$. Soit $\pi_0$ le plus grand diviseur de $N$ pour lequel $u$ poss\`ede une racine 
$\pi_0$-i\`eme dans $A$. Alors $X(u^{1/N})$ poss\`ede $\pi_0$ composantes connexes. 
\end{prop}

\begin{proof} On se ram\`ene facilement au cas $\pi_0=1$, et l'on veut montrer que 
$A[T]/(T^N-u)$ est int\`egre dans ce cas. Puisque $T^N-u$ est unitaire, 
$A[T]/(T^N-u)$ s'injecte dans $M[T]/(T^N-u)$, o\`u $M$ est le corps des fractions de 
$A$ (notons que $A$ est int\`egre par hypoth\`ese). Il suffit donc de montrer que 
$M[T]/(T^N-u)$ est un corps, i.e. que $T^N-u$ est irr\'eductible dans $M[T]$. Soit 
$P\in M[T]$ un facteur irr\'eductible unitaire et $k$ son degr\'e. Soit $t$ une racine $n$-i\`eme de 
$u$ dans une cl\^oture alg\'ebrique de $M$. Alors $T^N-u=\prod_{\zeta\in \mu_N(C)} (T-t\zeta)$, donc $P$ est un produit de $k$ facteurs de la forme $T-\zeta t$. Puisque $P(0)\in M$, on obtient $t^k\in M$. Puisque $(t^k)^N=u^k\in A^*$ et $A$ est normal, on en d\'eduit que $t^k\in A^*$ et donc $t^{{\rm pgcd}(N,k)}\in A^*$. L'hypoth\`ese 
$\pi_0=1$ force alors ${\rm pgcd}(N,k)=N$ et donc $k=N$ et $P=T^N-u$.

\end{proof}

Nous nous intéressons maintenant à des $\mu_N$-torseurs plus généraux sur $X$ que nous supposons muni d'une action d'un groupe $G$. Un $\mu_N$-torseur $\TC \to X$ est $G$-invariant si pour tout $g$ dans $G$, on a un isomorphisme de $\mu_N$-torseurs $g^{-1}\TC \iso \TC$. Le groupe des torseurs $G$-invariants s'identifie \`a $\het{1}(X, \mu_N)^G$. Nous dirons que $\TC$ est $G$-\'equivariant si les isomorphismes $g^{-1}\TC \iso \TC$ induisent une action de $G$ sur $\TC$ qui commute avec le rev\^etement. On a deux notions d'\'equivalence sur les torseurs $G$-\'equivariants. $\TC$ et $\TC'$ sont faiblement \'equivalents s'il existe un isomorphisme de $\mu_N$-torseurs $\TC \iso \TC'$ et fortement \'equivalents si on peut de plus le supposer $G$-\'equivariant. On note ${\rm Tors}^{(G)}(X)$ l'ensemble des $\mu_N$-torseurs $G$-\'equivariants \`a \'equivalence forte pr\`es.  Le r\'esultat suivant explicite le lien entre ces notions.

\begin{prop}
\label{propgequiv}
On a une suite exacte 
\[ 0 \to \hgal{1}(G, \mu_N(X)) \to  {\rm Tors}^{(G)}(X) \xrightarrow{\beta} \het{1}(X, \mu_N)^G \xrightarrow{\gamma} \hgal{2}(G, \mu_N(X)).\]
\end{prop} 

\begin{proof}
D\'ecrivons d'abord les fl\`eches $\beta$ et $\gamma$. Le morphisme $\beta$ est l'oubli de l'action de $G$. Pour $\gamma$, prenons $\pi : \TC \to X$ un torseur $G$-invariant et des isomorphismes $\rho_g : g^{-1}\TC \to \TC$. On consid\`ere 
\[ (g,h) \in G^2 \mapsto \rho_{gh} \circ (\rho_h \circ h^{-1} \rho_g)^{-1}\in\aut_X(\TC)=\mu_N(X).\] 
C'est un cocycle d'ordre 2 qui est un cobord si et seulement si les isomorphismes $\rho_g$ peuvent \^etre modifi\'es pour induire une action de $G$ sur $\TC$. Cela prouve l'exactitude en $ \het{1}(X, \mu_N)^G$. 

Pour terminer la preuve, il suffit de d\'ecrire le noyau de $\beta$. Ce groupe classifie les actions de $G$ qui commutent \`a $\mu_N$ sur le torseur trivial $\pi : X(1^{1/N}) \to X$. Prenons $t \in \Of^*( X(1^{1/N}))$ tel que $t^N=1$, on a la d\'ecomposition en parties isotypiques $\pi_*\Of_{X(1^{1/N})}= \bigoplus t^i \Of_X$. Une action de $G$ sera $\Of_X$-lin\'eaire et sera donc d\'etermin\'ee par sa restriction sur $t$. Si de plus cette derni\`ere commute avec celle de $\mu_N$, on pourra trouver des fonctions inversibles $\lambda_g$ pour tout $g$ telles que $g.t= \lambda_g t $. Mais on a $(g.t)^N=1$ donc $g \mapsto \lambda_g$ d\'efinit un cocycle de degr\'e $1$ de $\mu_N(X)$ qui est un cobord si et seulement si l'action est triviale. Ceci prouve l'exactitude \`a gauche.  
\end{proof}

\subsection{Récapitulatif des résultats obtenus pour $\H^d_K$\label{sssectionrecaphdk}}

Nous souhaitons ici rappeler les résultats de \cite{J1} et les réécrire de façon à pouvoir les utiliser dans la partie suivante. 

On a introduit deux recouvrements admissibles croissants par des ouverts admissibles 
$\H^d_K =\uni{\mathring{U}_n}{n> 0}{}=\uni{\bar{U}_n}{n \ge 0}{}$. On a aussi noté $\bar{U}_{n,L}=\bar{U}_{n}\hat{\otimes}_K L$, $\mathring{U}_{n,L}=\mathring{U}_{n} \hat{\otimes}_K L$ et $\H^d_L=\H^d_K\hat{\otimes}_K L$. Chacun de ces ouverts est stable sous l'action de $\gln_{d+1}(\OC_K)$.  


\begin{notation}

On pose $\Z\left\llbracket \HC\right\rrbracket^0=\limp_n \Z[\HC_n]^0$ avec $\Z[\HC_n]$  le $\Z$-module libre engendr\'e par $\HC_n$ et $\Z[\HC_n]^0$ le sous-module des fonctions de masse totale nulle\footnote{ie. les éléments $\sum_{h\in \HC_n}a_h\delta_h$ tels que $\sum_{h\in \HC_n}a_h=0$.}. Pour $l$ un entier premier à $p$ (non nécessairement premier), on définit de même les modules $\Z/l\Z\left\llbracket \HC\right\rrbracket^0$ et $\Z/l\Z[\HC_n]^0$.

\end{notation} Notons que $\Z\left\llbracket \HC\right\rrbracket^0$ et $\Z/l\Z\left\llbracket \HC\right\rrbracket^0$ admettent une action de $G$ (induite par celle   sur $\HC$) et qu'ils sont isomorphes aux duaux algébrique ${\rm St}_1 (\Z)^\vee$ et ${\rm St}_1 (\Z/l\Z)^\vee$ des représentations Steinberg associées. 

\begin{rem}\label{remfracrat}

Le choix d'un système de représentants $S_n\subset \OC_K^{d+1}$ unimodulaire de $\HC_n$ permet d'identifier le module $\Z[\HC_n]^0$ aux sous-groupes \[ \left\langle \frac{l_a}{l_b} :a\neq b\in S_n\right\rangle_{\Z\modut}\]
de $\Of^*(\H^d_{ L})$. Sous ce choix, nous pourrons voir les éléments de $\Z\left\llbracket \HC\right\rrbracket^0$ comme des suites de produit homogène de degré $0$ de formes linéaires que nous appellerons fractions rationnelles compatibles. 

De même, pour $l$ un entier premier à $p$, on peut écrire  les éléments de $\Z/l\Z[\HC_n]^0$ sous la forme \[\prod_{a\in S_n}l_a^{\beta_a}\ \ \ \text{avec } \beta_a\in \Z/l\Z \et \sum \beta_a = 0\]
que nous appellerons par abus fraction rationnelle réduite. On peut les interpréter comme des éléments de $\Of^*(\H^d_{ L})/\Of^*(\H^d_{ L})^l\subset \het{1}(\H^d_{ L},\mu_l)$ via la flèche de Kummer. De même, nous verrons les éléments de $\Z/l\Z\left\llbracket \HC\right\rrbracket^0$ comme des suites de fractions rationnelles réduites compatibles.  
\end{rem}

 Nous avons montré :
\begin{theo}[\cite{J1} théorèmes 6.8, 6.11, 7.1, 8.1 et proposition 8.2, \cite{J2} théorème 2.1]\label{theocohoanhdk}
Pour toute extension complète $L$ de $K$.
\begin{enumerate}
\item L'espace symétrique de Drinfeld $\H^d_{ L}$ et les affinoïdes $\bar{U}_{n,L}$ sont $\G_m$-acycliques en topologie analytique. En particulier,  on a $\pic (\H^d_{ L})=\pic (\bar{U}_{n,L})=0$.  
\item Les identifications dans \ref{remfracrat} induisent  des  suites exactes  $\gln_{d+1}(\OC_K)$-\'equivariantes \[ 0\to L^*\Of^{**} (\bar{U}_{n,L})\to \Of^* (\bar{U}_{n,L})\to \Z[\HC_{n+1}]^0\to 0,\]\[0\to L^*\to\Of^*(\H^d_{ L})\to\Z\left\llbracket \HC\right\rrbracket^0 \to 0.\] 
\item Soit $l$ premier à $p$, On dispose d'un isomorphisme et deux suites exactes
\[\het{1}(\bar{U}_{n,L}, \mu_l)\cong\het{1}(\mathring{U}_{n+1,L},\mu_l),\]
\[0\to L^*/(L^*)^l\fln{\kappa}{}\het{1}(\H^d_{ L},\mu_l)\to \Z/l\Z\left\llbracket \HC\right\rrbracket^0\to 0, \] \[ 0\to L^*/(L^*)^l\fln{\kappa}{}\het{1}(\bar{U}_{n,L}, \mu_l)\to\Z/l\Z[\HC_{n+1}]^0\to 0.\]

\item Les surjections précédentes s'inscrivent dans des diagrammes   commutatifs :
\[ \xymatrix{
\Of^*(\bar{U}_{n,L}) \ar[r]^-{\kappa} \ar[d] & \het{1}(\bar{U}_{n,L}, \mu_l)\ar[d] \\
 \Z[\HC_{n+1}]^0 \ar[r]_-{ } &  \Z/l\Z[\HC_{n+1}]^0}  \;\;\;\;\xymatrix{
\Of^*( \H^d_{K,L}) \ar[d]\ar[r]^-{\kappa} & \het{1}(\H^d_{ L}, \mu_l)\ar[d] \\
 \Z\left\llbracket \HC\right\rrbracket^0  \ar[r]_-{ } & \Z/l\Z\left\llbracket \HC\right\rrbracket^0 }.\]
\end{enumerate}
\end{theo}

\begin{proof}
Le point 2. a été montré dans \cite[Théorèmes 6.11, 7.1]{J1} et en partie dans \cite[Théorème 2.1]{J2}. Pour le premier point, nous renvoyons à \cite[Théorèmes 6.8,  7.1]{J1}.

Les points 3. et 4. découlent directement des points 1. et 2. par suite exacte de Kummer et théorème de comparaison \cite[Proposition 8.2]{J1}.
\end{proof}

\begin{rem}\label{remscin}

Modulo le choix d'une famille de systèmes de représentants convenables $(S_n)_n$ de $(\HC_n)_n$, on peut, d'après la remarque \ref{remfracrat} et \cite[Remarque 2.8]{J2}, scinder les suites exactes du point 2. de \ref{theocohoanhdk}. Même si ce scindage dépend des choix de représentants, il permet de donner une représentation explicite aux fonctions inversibles. Par contre, l'action de $G$ ne respecte pas la décomposition en produit obtenu car elle permute les systèmes de représentants "convenables". 

\end{rem}

Au vu de la remarque précédente, on fixe un jeu de systèmes de représentants et nous allons décrire plus concrètement les isomorphismes des points 2. et 3. de \ref{theocohoanhdk} grâce à ces scindages. En premier lieu, toute fonction inversible $u$ sur l'espace symétrique est équivalente à la donnée d'une constante $\lambda$ et d'une famille de fractions rationnelles compatibles $(u_n)_n$ qui sont  caractérisées par la congruence (on a confondu les éléments $u_n \in \Z[\HC_n]^0$ avec les fonctions inversibles associées dans $\Of^*(\H^d_{ L})$ qui découlent des  constructions de \ref{remfracrat})  \[u\equiv \lambda u_n\pmod{\Of^{**}(\bar{U}_{n-1,L})}.\] De plus, on peut exprimer $u$ en fonction de ces deux données sous la forme $u=\lim_n \lambda u_n$ (contrairement aux éléments $\lambda u_n$, la limite $u$ ne dépends du scindage).

 De même, tout $\mu_l$-torseur $\TC\to \H_{ L}^d$ est caractérisé par la donnée d'une constante $\lambda\in L^*/(L^*)^l$ et d'une famille de fractions rationnelles réduites compatibles  $(v_n)_n$. Ainsi, tout torseur $\TC\to \H_{ L}^d$ provient d'une fonction inversible $\tilde{v}=(\lambda \tilde{v}_n)$ via la flèche de Kummer et \[ \TC|_{\bar{U}_{n-1,L}}=\bar{U}_{n-1,L}((\lambda_n \tilde{v}_n)^{1/l})\] avec $ \tilde{v}_n$ un relevé de  $v_n$ dans  $\Of^*(\H^d_{ L})$ (cf \ref{remfracrat}). Enfin, tout $\mu_l$-torseur sur $\bar{U}_{n-1,L}$ se prolonge de manière unique sur $\mathring{U}_{n,L}$ et admet même un prolongement global (qui n'est pas unique cette fois-ci).

\subsection{Equation globale pour le torseur $\Sigma_1$\label{sssectioneqsigma1}}

En guise d'application, nous allons donner une description explicite du premier rev\^etement de l'espace sym\'etrique de Drinfeld $\Sigma^1$. 
On pose $N=q^{d+1}-1=(q-1)\tilde{N}$, le revêtement $\Sigma^1$ a groupe de galois $\mu_N$. D'après ce qui précède, \[\Sigma^1=\H^d_{ \breve{K}}((\varpi^k u)^{1/N})\] avec $\varpi^k u=(\varpi^k u_n)_n=(\varpi^k \prod_{H\in S_n}l_H^{\alpha_H})_n$  un \'el\'ement de \[\het{1}( \H_{ \breve{K}}^d, \mu_N)=\breve{K}^*/(\breve{K}^*)^N\times \Z/N\Z\left\llbracket \HC\right\rrbracket^0=\varpi^{\Z/N\Z}\times \Z/N\Z\left\llbracket \HC\right\rrbracket^0.\] 

\begin{rem}

Notons que l'action du groupe $\gln_{d+1} (\OC_K)$ respecte la décomposition  en produit ci-dessus de $\het{1}( \H_{ \breve{K}}^d, \mu_N)$. Plus précisément, le groupe agit trivialement sur le premier facteur $\varpi^{\Z/N\Z}$ et de manière naturelle sur le second facteur $\Z\left\llbracket \HC\right\rrbracket^0$ (ie. via l'action sur $\HC$). 

Justifions que l'action sur $\varpi^{\Z/N\Z}$ est la bonne. Prenons  $v=(\varpi^k v_n)_n\in \het{1}( \H_{ \breve{K}}^d, \mu_N)$, ainsi  que $g\in \gln_{d+1} (\OC_K)$ et décomposons \[g\cdot v =(\varpi^{k_g} g\cdot v_n)_n.\] N'importe quel relevé $\tilde{v}\in \Of^*(\H^d_{ L})$ de $v$ vérifie $\tilde{v}\in \varpi^k \tilde{v}_1\Of^{**}(\H^d_{ L,s})$ sur le sommet standard $s$ (par abus, on voit $k$ dans $\Z$) et on a $\left\Vert  v\right\Vert_{\H^d_{ L,s}}=|\varpi^k |$ d'après les équations définissant le tube au-dessus du sommet $\H^d_{ L,s}$ (cf \ref{paragraphhdksimpfer}). De plus, sur ce sommet, l'élément agit $g$ par isométrie et en en déduit \[\left\Vert v\right\Vert_{\H^d_{ L,s}}=\left\Vert g\cdot v\right\Vert_{\H^d_{ L,s}}=|\varpi^{k_g}|\] d'où $k=k_g$.

\end{rem}

Nous allons d\'eterminer les quantit\'es $k,\alpha_H\in \Z/N\Z$ pour décrire explicitement $\Sigma^1$. L'observation clé est de voir que le revêtement $\pi : \Sigma^1\to\H^d_{ \breve{K}}$ est $\gln_{d+1}(\OC_K)$-\'equivariant. Le résultat suivant montre qu'il n'y en a essentiellement qu'un seul.

\begin{lem}\label{leminth1}

Si on pose\footnote{Comme on a $|\HC_n|q^{n-1}(q-1)\in N\Z$, on  vérifie $(u_n^{(q-1)})_n\in\Z/N\Z\left\llbracket \HC\right\rrbracket^0$. On peut aussi les écrire sous la forme $u_n^{(q-1)}=\prod_{H\in S_n\backslash \{H_0\}} (\frac{l_H}{l_{H_0}})^{q^{n-1}(q-1)}$.} $  u=(u_n)_n :=(  \prod_{H\in S_n} l_H^{q^{n-1}})_n\in   \Z/\tilde{N}\Z\left\llbracket \HC\right\rrbracket^0$ alors  \[\het{1}(\H^d_{ \breve{K} }, \mu_N)^{\gln_{d+1}(\OC_K)}=\varpi^{\Z/N\Z}\times \langle \kappa(u^{(q-1)}) \rangle_{\Z/N\Z\modut}\] et un torseur de $\het{1}(\H^d_{ \breve{K} }, \mu_N)^{\gln_{d+1}(\OC_K)}$ est déterminé par sa restriction à $\H^d_{ \breve{K},s}$. 

\end{lem}

\begin{proof}
Soit $\varpi^k v=(\varpi^k v_n)_n=(\varpi^k\prod_{H \in \HC_n} l_H^{\alpha_H})_n$ un torseur de $\het{1}(\H^d_{ \breve{K} }, \mu_N)^{\gln_{d+1}(\OC_K)}$. 
L'action de $\gln_{d+1} (\OC_K)$ sur $\het{1} (\H^d_{ \breve{K}}, \mu_N)$ est induite par l'action naturelle sur chaque $\HC_n$ qui est transitive. Ainsi, on a $\alpha_H= \alpha_{H'}=: \alpha_n$ pour tout $H$, $H'$ dans $\HC_n$. Les fractions rationnelles réduites $v_n$ doivent former une famille compatible et être homog\`enes de degr\'e $0$, on a donc les relations $\alpha_{n}= \frac{|\HC_{n+1}|}{|\HC_n|}\alpha_{n+1}$ et $| \HC_n |\alpha_n \equiv 0 \pmod N$.

Calculons la quantité $|\HC_{n}|$. On a une identification \[\HC_n=\{\text{vecteur unimodulaire de } K^{d+1} \text{ mod } \varpi^n\}/(\OC_K/\varpi^n\OC_K)^*. \]
Le premier est de cardinal \[|(\OC_K/\varpi^n\OC_K)^{d+1}\backslash (\varpi\OC_K/\varpi^n\OC_K)^{d+1}|=q^{n(d+1)}-q^{(n-1)(d+1)}=Nq^{(n-1)(d+1)}\] et le second $(q-1)q^{n-1}$ d'où 
$|\HC_n|=\tilde{N}q^{(n-1)d}$.
On en d\'eduit alors
\[ \alpha_n = q^d \alpha_{n+1} \et \alpha_n \in (q-1) \Z/ N \Z = \mathrm{Ann}(| \HC_n | )=\mathrm{Ann}(\tilde{N}).\]
Ainsi, pour tout $n$, on a $\alpha_n= q^{n-1} (q-1) \tilde{\alpha}_1$ ie. $v=u^{(q-1)\tilde{\alpha}_1}$ avec $u$ le torseur introduit dans l'énoncé. La donné de $\tilde{\alpha}_1$ et de $k$ est équivalente à $\varpi^k v_1$ et donc à la restriction de $\varpi^k v$ sur $\bar{U}_{0,\breve{K}}=\H^d_{ \breve{K},s}$. D'après la discussion précédente, $\tilde{\alpha}_1$ et  $k$ déterminent aussi $\varpi^k v$ et donc le torseur $\TC$ ce qui termine la preuve.
\end{proof}

Par \ref{theoniventiersigmasommet}, on sait de plus  : 
 \begin{equation}\label{eqsigma1}
\Sigma^1|_{\bar{U}_{0,\breve{K}}}= \Sigma^1_{s}=\H^d_{ \breve{K},s}((\varpi \prod_{H \in \HC_1} (l_H)^{q-1})^{1/N})=\H^d_{ \breve{K},s}((\lambda u_1^{(q-1)})^{1/N})
 \end{equation}
ce qui entraîne par la discussion précédente :

\begin{theo}\label{theoeqsigma1} On reprend l'élément $u=( u_n)_n\in   \Z/N\Z\left\llbracket \HC\right\rrbracket^0$ définie dans \ref{leminth1}. On a 
\[  \Sigma^1=\H^d_{ \breve{K}}((\varpi u^{(q-1)})^{1/N}). \]
\end{theo}

\begin{rem}

La preuve montre plus précisément :
\[ \Sigma^{1}|_{\mathring{U}_{n,\breve{K}}} = \mathring{U}_{n,\breve{K}}((\varpi u_{n }^{(q-1)})^{1/N})\et\Sigma^{1}|_{\bar{U}_{n,\breve{K}}} = \bar{U}_{n,\breve{K}}((\varpi u^{(q-1)}_{n+1})^{1/N}). \]

\end{rem}

On retrouve aisément le résultat bien connu par  \ref{proppio0}
\begin{coro}\label{coropi0}

Les revêtements $\Sigma^1$ et $\Sigma^1_{|\bar{U}_{n,\breve{K}}}$ ont $q-1$ composantes connexes géométriques.

\end{coro}

\begin{rem}\label{remequiv}

Les constantes $\kappa(\lambda)$ et le torseur $\kappa(u^{(q-1)}) $ engendrent $\het{1}(\H^d_{ \breve{K} }, \mu_N)^{\gln_{d+1}(\OC_K)}$ et sont faiblement équivariant, la suite exacte écrite dans \ref{propgequiv} devient :
\[0\to\F^*\to{\rm Tors}^{\gln_{d+1}(\OC_K)}(\H^d_{ \breve{K} })\to \het{1}(\H^d_{ \breve{K} }, \mu_N)^{\gln_{d+1}(\OC_K)}\to 0.\]

\end{rem}

\section{Calcul des fonctions inversibles du revêtement modéré\label{sssectionost}}

Comme précédemment, on note $N=q^{d+1}-1=(q-1)\tilde{N}$. Nous souhaitons calculer les fonctions inversibles de $\Sigma_L^1$. On suppose que $L/ \breve{K}$ est galoisienne et qu'elle contient une racine $N$-i\`eme de $\varpi$. Les autres cas se d\'eduiront en passant aux invariants sous le groupe de Galois. Donnons-nous une unité $u\in  \Of^*(\H^d_{ \breve{K}})$ tel que $\Sigma^{1}=\H^d_{ \breve{K}}((\varpi u^{q-1})^{1/N})$ (cf \ref{theoeqsigma1}). Ainsi, $\Sigma^1_L=\H_{ L}^d((u^{q-1})^{1/N})$ par hypothèse sur $L$.

On fixe $t$ dans $\Of^*(\Sigma^1_L)$ une racine $N$-i\`eme de $u^{q-1}$ et on pose $t_0=\frac{t^{\tilde{N}}}{u}$. Pour tout $a\in \mu_{q-1}(K)$, on introduit le polynôme interpolateur de Lagrange $\Lf_a (X)$ ie. le polynôme unitaire de degré $q-2$ qui vaut $1$ en $a$ et $0$ sur $\mu_{q-1}(K)\backslash\{a\}$. Les composantes de $\Sigma^1_L$ sont alors en bijection avec la famille d'idempotents $\{ \Lf_a (t_0) \}_{a\in \mu_{q-1}(K)}$.

Enfin, si $Y\to X$ est un revêtement de groupe de galois cyclique, nous introduisons la norme galoisienne ${\rm Nrm}: \Of(Y) \to \Of(X)$ avec ${\rm Nrm}(f)= \prod_g g \cdot f$ o\`u $g$ parcourt les automorphismes de $Y/X$.

\begin{theo}\label{theogmsigma1}
Reprenons les notations pr\'ec\'edentes, toute fonction inversible $v$ de $\Sigma^1_L$ admet une unique \'ecriture sous la forme \[v= \sum_{a\in \mu_{q-1}(K)} t^{j_a}v_a \Lf_a (t_0)\] o\`u $(v_a)_a$ sont des fonctions inversibles de $\H^d_{ L}$ et $0\le j_a\le  \tilde{N}-1$.  

De même, si $T \subset \BC\TC$ est un sous-complexe simplicial connexe, toute fonction inversible $v$ de $\Sigma^1_{L,T}$ admet une  \'ecriture sous la forme \[v= \sum_{a\in \mu_{q-1}(K)} t^{j_a}v_a (1+f_a) \Lf_a (t_0)\] o\`u $f_a$ est topologiquement nilpotente sur $\Sigma^1_{L,T}$,  $(v_a)_a$ sont des sections inversibles de $\H^d_{ L}$ et $0\le j_a\le  \tilde{N}-1$.  
\end{theo}

\begin{rem}
Dit autrement, sur chaque composante connexe géométrique  du premier rev\^etement, les fonctions inversibles du premier rev\^etement sont engendr\'ees par $\Of^*(\H^d_{ L})$ et par $t$.
\end{rem}

Comme on peut raisonner sur chaque composante connexe, nous nous concentrerons plut\^ot sur le rev\^etement $\tilde{\Sigma}^1_L=\H^d_{ L}(u^{1/ \tilde{N}})$. Par abus, nous \'ecrirons encore $t$ une racine $\tilde{N}$-i\`eme de $u$. 

\subsection{Lemmes techniques}

Nous aurons besoin des deux r\'esultats techniques suivants :

\begin{lem}[Principe du maximum] \label{lemmaxprinc}
Soit $T \subset \BC\TC$ un sous-complexe simplicial connexe, une fonction $f=\sum_{i=0}^{\tilde{N}-1}f_i t^i$ de $\tilde{\Sigma}^1_{L,T}$ est \`a puissance born\'ee (respectivement topologiquement nilpotente) si et seulement si, pour tout $i$, $u^i f^{\tilde{N}}_i$ l'est sur $\H^d_{L,T}$ 
si et seulement si pour tout sommet $s$ de $T$, $f$ l'est sur $\tilde{\Sigma}^1_{L,s}$.
\end{lem}
Pour  $H$ un hyperplan $\F$-rationnel de $\P^d_\F$, on fixe $l_H : z \mapsto \sum_i a_i z_i$ une forme lin\'eaire telle que $H= \ker(l_H)$. On dit que $H= \infty$ si $l_H(z)=z_d$. Dans $\F[X_0, \dots, X_{d-1}]$, on note encore \[l_H=a_0X_0+ \dots+a_{d-1} X_{d-1}+a_d\] quand $H\neq \infty$. Cela revient à restreindre $\frac{l_H(z)}{z_d}$ à $\P^d_\F\backslash \infty=\spec (\F[X_0, \dots, X_{d-1}])$.

\begin{lem}\label{lemkernrmbarsigma1}
On consid\`ere le rev\^etement \[ \F[X_0, \dots, X_{d-1}, \frac{1}{P}] \to \F[X_0, \dots, X_{d-1}, \frac{1}{P}, t]/(t^{\tilde{N}}-P) \] o\`u $P= \prod_{H \in \P^d(\F_q) \backslash \infty} l_H$. Soit $v$ dans $\F[X_0, \dots, X_{d-1}, \frac{1}{P}, t]/(t^{\tilde{N}}-P)$, alors ${\rm Nrm}(v)=1$ si et seulement si $v$ est une racine $\tilde{N}$-i\`eme de l'unit\'e. 
\end{lem}

\begin{proof}[Preuve du lemme \ref{lemmaxprinc}]
Prouvons la première équivalence. Un sens est \'evident, prouvons l'autre. 
Soit $f= \sum_i f_i t^i$ une section sur $\tilde{\Sigma}^1_{L,T}$ avec $f_i$ une section de $\H^d_{ L,T}$. Chaque translat\'e de $f$ par le groupe de Galois du rev\^etement est \`a puissance born\'ee (respectivement topologiquement nilpotente) si et seulement si $f$ l'est. Mais la projection $f \mapsto f_i t^i$ est une combinaison lin\'eaire \`a coefficients dans $\Z[\frac{1}{\tilde{N}}]\subset\Z_p$ ($\tilde{N}$ premier à $p$) des translat\'es de $f$ et on en d\'eduit que $f$ est \`a puissance born\'ee (respectivement topologiquement nilpotente) si et seulement si chaque $f_it^i$ l'est si et seulement si chaque $f_i^{\tilde{N}} u^i=(f_i t^i)^{\tilde{N}}$ l'est. On s'est ramené au niveau 0 car $f_i^{\tilde{N}} u^i$ est dans $\Of(\H^d_{ L,T})$. 

Intéressons-nous maintenant à la seconde équivalence. D'après la discussion précédente, il suffit de montrer que, pour une fonction $f$ de $\H^d_{ L,T}$, le maximum de la norme spectrale est atteint sur un sommet. On peut alors raisonner sur 
$T= \sigma$ un simplexe. 
Même si la preuve qui va suivre est auto-contenue, elle s'inspire d'arguments que l'on peut retrouver dans \cite[Proposition 3.8.]{J2}. La flèche suivante est un morphisme pluri-nodal au sens de \cite[Definition 1.1]{ber7}
\[\OC_L\to\hat{A}_\sigma =\OC_L\langle X_0,\cdots, X_d, \frac{1}{P_\sigma} \rangle /(X_{d_0}\cdots X_{d_k}-\varpi))\] avec $k$ la dimension de $\sigma$. D'après \cite[Proposition 1.4]{ber7}, on a $\hat{A}_\sigma=\Of^+(\H^d_{ L,\sigma})$ et la norme spectrale sur $\H^d_{ L,\sigma}$ prend ses valeurs dans $|L^*|$. Ainsi, toute fonction $f\in {A}_\sigma$ se décompose sous la forme $\lambda \tilde{f}$ avec $\lambda$ une constante telle que $\left\Vert f\right\Vert_{\H^d_{ L,\sigma}}=|\lambda|$. En particulier, on observe les différentes relations $\left\Vert \tilde{f}\right\Vert_{\H^d_{ L,\sigma}}=1$ et $\Of^{++}(\H^d_{ L,\sigma})=\mG_L \hat{A}_\sigma$. Revenons à la preuve, il suffit de montrer que, pour une fonction $f=\lambda \tilde{f}\in {A}_\sigma$, on a $\tilde{f}\in \hat{A}_s\backslash \mG_L \hat{A}_s$ sur un sommet $s\in\sigma$ bien choisi. En effet dans ce cas, on aurait la suite d'égalité
\[\left\Vert f\right\Vert_{\H^d_{ L,s}}=|\lambda|=\left\Vert f\right\Vert_{\H^d_{ L,\sigma}}\]
qui montre que le maximum de la norme spectrale est atteint sur un sommet. Pour cela, on observe que la fonction $\tilde{f}$ est non nulle dans $\bar{A}_\sigma=\hat{A}_\sigma/ \mG_L \hat{A}_\sigma$. Les $(k+1)$ tubes au-dessus des sommets de $\sigma$ en fibre spéciale sont décrits de la manière suivante
\[\H^d_{ \kappa_L,s_r} =\spec (\bar{A}_\sigma [(\frac{1}{X_{d_{i_t}}})_{t\neq r}])=\spec (\bar{A}_{s_r})\]
quand $r$ parcourt $0,\cdots, k$. Ainsi, l'espace $\H^d_{ \kappa_L,s_r}$ est un ouvert dense de $V(X_{d_{i_r}})$ par irréductibilité. On en déduit  la densité de $\bigcup^k_{r=0}\H^d_{ \kappa_L,s_r}$ dans $\bigcup^k_{r=0}V(X_{d_{i_r}})=\H^d_{\kappa_L,\sigma}$. La fonction $\tilde{f}$ est donc non nulle sur un des ouverts $\H^d_{ \kappa_L,s_r}$ ce qui conclut la preuve.

\end{proof}

\begin{proof}[Preuve du lemme \ref{lemkernrmbarsigma1}]
Nous introduisons d'abord la notion d'ordre en un hyperplan. Prenons $H$ un hyperplan $\F$-rationnel. Comme $\F [X_0, \dots, X_{d-1}]$ est factoriel, tout polyn\^ome $Q$ s'\'ecrit comme un produit $\prod_i P_i^{\alpha_{P_i}}$ et on pose $v_H(Q)= \alpha_{l_H}$. On pose de m\^eme $v_{\infty}(Q)= -{\rm deg}(Q)$ o\`u ${\rm deg}$ d\'esigne le degr\'e total. Ces applications sont multiplicatives et v\'erifient l'in\'egalit\'e ultram\'etrique. Elles se prolongent de mani\`ere unique en des applications multiplicatives sur $\F[X_0, \dots, X_{d-1}, \frac{1}{P}]$. Nous expliquons comment les prolonger \`a $\F[X_0, \dots, X_{d-1}, \frac{1}{P}, t]/(t^{\tilde{N}}-P)$.

Prenons $f= \sum_i t^i f_i$, on pose dans un premier temps \[v_H(t^i f_i)= \frac{i}{\tilde{N}} + v_H(f_i)\et v_{\infty}(t^i f_i)=\frac{i(1-\tilde{N})}{\tilde{N}} + v_{\infty}(f_i).\] On remarque que cette d\'efinition est coh\'erente avec les identit\'es $v_H(P)= 1= \tilde{N} v_H(t)$ et $v_{\infty}(P)= 1-\tilde{N}= \tilde{N} v_{\infty}(t)$. De mani\`ere g\'en\'erale, on pose $v_H(f)= \min(v_H(t^if_i))$ et $v_{\infty}(f)= \min(v_{\infty}(t^if_i))$. Notons que $v_H(t^if_i)=\frac{i}{\tilde{N}}$ (idem pour $v_{\infty}$) dans $\Z[1/\tilde{N}]/\Z$ et sont donc tous différents dans $\Z[1/\tilde{N}]$ quand $i$ varie. Le minimum est donc atteint pour un unique indice $i$.

Les valuations $v_H$ et $v_{\infty}$ v\'erifient encore l'in\'egalit\'e ultram\'etrique et nous allons montrer qu'elles sont multiplicatives. 
Prenons $f= \sum f_i t^i$ et $g= \sum g_j t^j$, donnons-nous $i_0$, $j_0$ les uniques entiers pour lesquels $v_H(f)= v_H(f_{i_0}t^{i_0})$ et $v_H(g)= v_H(g_{j_0}t^{j_0})$. Par in\'egalit\'e ultram\'etrique, il suffit de montrer en d\'eveloppant le produit $fg$ que le minimum de $v_H(f_i t^i g_j t^j)$ est atteint uniquement pour le couple $(i_0, j_0)$. Si par exemple $i\neq i_0$, on a $v_H(f_i t^i) > v_H(f_{i_0} t^{i_0})$ car les valuations sont toutes diff\'erentes. On en d\'eduit que $v_H(f_it^i g_j t^j) > v_H(f_{i_0} t^{i_0} g_{j_0} t^{j_0})$. On raisonne de m\^eme pour $j \neq j_0$.  
On v\'erifie ais\'ement que $v_H$ et $v_{\infty}$ sont invariants sous le groupe de Galois du rev\^etement et on observe les \'egalit\'es $v_H( {\rm Nrm}(f))= \tilde{N} v_H(f)$ et $v_{\infty}( {\rm Nrm}(f))= \tilde{N} v_{\infty}(f)$.

Revenons \`a l'\'enonc\'e et prenons $u= \sum u_i t^i$ tel que de norme galoisienne $1$. D'apr\`es ce qui pr\'ec\`ede, $v_H(u)=0= v_{\infty}(u)$. On en d\'eduit $v_H(u_0)=0$ pour tout $H$ et $v_H(u_i) \ge 0$ pour $i \neq 0$. Ainsi, les $u_i$ sont des polyn\^omes. De m\^eme, $v_{\infty}(u_0)=0$ et $v_{\infty}(u_i) > 0$ pour tout $i \neq 0$. Ainsi, $u_i=0$ pour $i \neq 0$ et $u=u_0$ est une constante de norme galoisienne $1$; $u$ est donc une racine $\tilde{N}$-i\`eme de l'unit\'e. 
\end{proof}

\subsection{Conséquences}

Nous tirons des lemmes précédents les résultats suivants

\begin{coro}\label{corokernrmsigma1}
Soit $f$ une fonction sur $\tilde{\Sigma}^1_L$ telle que ${\rm Nrm}(f)=1$ alors $f$ est une racine $\tilde{N}$-i\`eme de l'unit\'e.  
\end{coro}

\begin{coro}\label{corokernrmsigma1T}
Si $T \subset \BC\TC$ est un sous-complexe simplicial connexe, $f \in \Of^*(\tilde{\Sigma}_{L,T}^1)$ telle que ${\rm Nrm}(f)=1$ alors \[f\in\Of^*(\H^d_{ L,T})(1+\Of^{++}(\tilde{\Sigma}_{L,T}^1)). \] 
\end{coro}

\begin{proof}[Preuve de \ref{corokernrmsigma1}]
On prend $f=\sum^{\tilde{N}-1}_{i=0} f_i t^i$ de norme galoisienne $1$, on raisonne sommet par sommet grâce au lemme \ref{lemmaxprinc}. Fixons $s$ dans $\BC\TC_0$, le tube $\tilde{\Sigma}^1_s$ admet un mod\`ele entier lisse sur $\OC_L$ de fibre sp\'eciale $\spec(\F[X_0, \dots, X_{d-1}, \frac{1}{P}, t]/(t^{\tilde{N}}-P))$ \cite[2.3.8]{wa}. Comme la fibre sp\'eciale est int\`egre, la norme spectrale est multiplicative et \[\Vert {\rm Nrm}(f) \Vert_{\tilde{\Sigma}^1_{L,s}}=\Vert f \Vert^{\tilde{N}}_{\tilde{\Sigma}^1_{L,s}}=1.\] 
On en d\'eduit que $f$ est \`a puissance born\'ee sur tout $\tilde{\Sigma}^1_L$ par \ref{lemmaxprinc} et chaque $u^i f_i^{\tilde{N}}$ est constante toujours d'apr\`es \ref{lemmaxprinc} et \cite[lemme 3]{ber5}. Si cette constante est non nulle quand  $i\neq 0$, alors $f_i$ est inversible sur $\H^d_{ L}$ et on a $u^i\in L^*(\Of^*(\H^d_{ L}))^{\tilde{N}}\subset (\Of^*(\H^d_{ C}))^{\tilde{N}}$ ce qui contredit l'hypothèse de connexité  géométrique (cf \ref{proppio0}) de $\tilde{\Sigma}^1_L$. Donc $f_i=0$ quand $i\neq 0$ et $f=f_0\in L^*$ de norme galoisienne $1$, c'est une racine $\tilde{N}$-ième de l'unité.
\end{proof}

\begin{proof}[Preuve de \ref{corokernrmsigma1T}]
Comme dans le raisonnement précédent, les restrictions de $f$ et $f^{-1}$ \`a chaque sommet sont de norme spectrale $1$. D'après \ref{lemmaxprinc}, on a $f, f^{-1}\in \Of^+(\tilde{\Sigma}^1_{ L,T})$ 
et il existe une racine de l'unit\'e $\lambda_s$ et une fonction $g_s$ dans $\Of^{++}(\tilde{\Sigma}^1_{ L,s})$ telles que $f= \lambda_s (1+g_s)$ au-dessus de chaque sommet $s$ d'apr\`es \ref{lemkernrmbarsigma1}. On écrit $f= \sum_i t^i f_i$, si bien que sur chaque sommet $s$, on a  $f_0-\lambda_s\in \Of^{++}(\H^d_{ L,s})$ et $t^i f_i\in \Of^{++}(\tilde{\Sigma}^1_{L,s})$ si $i\neq 0$. On a par principe du maximum, $f-f_0=\sum_{i\neq 0}t^i f_i\in \Of^{++}(\tilde{\Sigma}^1_{L,T})$.  Il s'ensuit que
\begin{align*}
f_0 & = f- (f-f_0) = f(1- f^{-1}(f-f_0))\in f\Of^{**}(\tilde{\Sigma}^1_{L,T})\cap\Of(\H^d_{ L,T})\subset (\Of^{+}(\H^d_{ L,T})^*.
\end{align*}
Ce dernier point entra\^ine que \[f= f_0(1+ f_0^{-1}(f-f_0))  \in f_0\Of^{**}(\tilde{\Sigma}^1_{L,T})\] et on a exhib\'e la d\'ecomposition voulue. 
\end{proof}

\subsection{Fin de la preuve}

Revenons \`a la preuve du th\'eor\`eme principal. Soit $v$ une section inversible de $\tilde{\Sigma}^1_L$ et consid\'erons $\frac{v^{\tilde{N}}}{{\rm Nrm}(v)}$. Elle est  de norme galoisienne 1, c'est une racine $\tilde{N}$-i\`eme de l'unit\'e $\zeta$ d'apr\`es \ref{corokernrmsigma1}. Ainsi on obtient \[ v^{\tilde{N}}=\zeta {\rm Nrm}(v) \in \Of^*(\H^d_{ L}). \] D'apr\`es cette relation, les translat\'es de $v$ par le groupe de Galois sont de la forme $\lambda v$ avec $\lambda$ une racine $\tilde{N}$-i\`eme de l'unit\'e par connexité  géométrique  de $\tilde{\Sigma}^1_L$. Ainsi, $v$ est un vecteur propre pour tous les automorphismes du rev\^etement. Mais les parties isotypiques sont de la forme $t^i \Of( \H^d_{ L})$ pour $0 \le i \le \tilde{N}-1$ donc il existe $w$ dans $\Of^*(\H^d_{ L})$ tel que $v= t^i w$. 

Prenons maintenant $v$ une section inversible $\tilde{\Sigma}^1_{L,T}$ et consid\'erons encore la section inversible $\frac{v^{\tilde{N}}}{{\rm Nrm}(v)}$ de norme galoisienne $1$. On l'\'ecrit $f_0(1+g)=f_0(1+ \tilde{g})^{\tilde{N}}$ comme dans \ref{corokernrmsigma1T}. On a alors \[ (\frac{v}{1+ \tilde{g}})^{\tilde{N}}= f_0 {\rm Nrm}(v) \in \Of^*(\H^d_{ L,T}). \] L'argument pr\'ec\'edent sur les parties isotypiques prouve l'existence de $i$ et de $w$ dans $\Of^*(\H^d_{ L,T})$ tel que \[ v=t^i w (1+ \tilde{g}). \]

\section{Description de $\XG[\Pi_D]$ au-dessus d'un simplexe maximal \label{secmax}}


Dans cette section, nous chercherons à approfondir la description de $\XG[\Pi_D]$ établie dans \ref{theoniventiersigmasommet}. En effet, nous avons seulement pu étudier ce schéma en groupe au-dessus d'un sommet de l'immeuble de Bruhat-Tits qui ne correspond qu'à quelques ouverts de $\H^d_{\OC_{\breve{K}}}$. Toutefois, on a le recouvrement par les simplexes maximaux \[\H_{  \OC_{\breve{K}}}^d=\bigcup_{\sigma\in\BC\TC_d} \H_{ \OC_{\breve{K}},\sigma}^d\] et il suffit de comprendre la restriction $\XG[\Pi_D]_\sigma$ de  $\XG[\Pi_D]$ sur chacun des ouverts $\H_{  \OC_{\breve{K}},\sigma}^d=\spf(\hat{A}_\sigma)$. Pour cela nous allons utiliser la classification des schémas de Raynaud rappelée dans \ref{ssectionrayn} et la compréhension de la fibre générique obtenue grâce à l'équation \ref{theoeqsigma1}.

Nous rappelons (voir \ref{sssectionhdkcalcexp}) que l'anneau $\hat{A}_{\sigma}$ est la complétion p-adique de
\[\OC_{\breve{K}}[x_0, \dots, x_d, \frac{1}{\prod_{i\in\Z/(d+1)\Z}\prod_{a\in R_i} P_{a}}]/(x_0 \dots x_d- \varpi)\]
avec $P_a=  1+ \tilde{a}_{i-1} x_{i-1}+ \dots + \tilde{a}_{i-d} x_{i-1} \dots x_{i-d}$ pour $a=(  \tilde{a}_0, \dots,  \tilde{a}_{i-1},1, \varpi\tilde{a}_{i+1} , \dots, \varpi\tilde{a}_d) \in R_i$

Le résultat que nous allons prouver est le suivant :

\begin{theo}\label{theoxpid}
Soit $\sigma \subset \BC\TC$ un simplexe maximal, on a 
\begin{equation}
\label{eqxpidsimp} 
\XG[\Pi_D]_{\sigma} \cong \spf(\hat{A}_{\sigma} [y_0, \dots, y_d]/ \langle y_i^q-u_ix_{d-i}y_{i+1} \rangle_{i \in \Z/(d+1)\Z} ) 
\end{equation}
avec 
\begin{equation}
\label{equi}
u_i=\frac{\prod_{\tilde{b}\in R_{d-1-i}} P_{\tilde{b}}}{ \prod_{\tilde{b}\in R_{d-i}} P_{\tilde{b}}} \in \hat{A}_{\sigma}^*.
\end{equation}  
\end{theo}

\begin{rem}\label{remui}
\begin{itemize}
\item Le premier point \eqref{eqxpidsimp} d\'ecoule de la description de la fibre sp\'eciale (cf \ref{propxpidlib}) et a d\'ej\`a \'et\'e donn\'e par \cite{teit2} en dimension $1$ et \cite[2.3, 2.4]{wa} en dimension sup\'erieure. La description des $u_i$ \eqref{equi} proviendra de la description de la fibre g\'en\'erique \ref{lemeqsigsig}. Dans le cas particulier où $d=1$ et $K=\Q_p$, on retrouve le résultat \cite[Corollary 6.7.]{pan}.
\item Avec ce choix de représentants pour la famille $(u_i)_i$ (cf \eqref{equi}), on a $u_0\cdots u_d=1$.
\end{itemize}

\end{rem} 

\subsection{Section globale de $\XG[\Pi_D]_\sigma$ et preuve de \ref{theoxpid} \eqref{eqxpidsimp}}

Pour simplifier, nous écrirons $\hat{B}_{\sigma}= \Of(\XG[\Pi_D]_{\sigma})$ et $B_{\sigma}:= \hat{B}_{\sigma}[\frac{1}{p}]$. Notre but ici est de montrer l'existence d'une classe de fonctions inversibles $(u_i)_i$ pour lesquelles l'anneau $\hat{B}_{\sigma}$ peut s'écrire
\[\hat{B}_{\sigma}=\hat{A}_{\sigma} [y_1, \dots, y_d]/ \langle y_i^q-u_ix_{d-i}y_{i+1} \rangle_{i \in \Z/(d+1)\Z}.\]
Nous donnerons une expression à ces fonctions dans la section suivante. Le résultat découle essentiellement de \ref{propxpidlib} et de
\begin{prop}\label{lempic}

$\pic (\H_{  \OC_{\breve{K}},\sigma}^d)=0.$

\end{prop}
En effet, supposons dans un premier temps cette annulation. Ainsi, on peut \'ecrire d'après \ref{propxpidlib}
\[ \XG[\Pi_D]_{\sigma} \cong \spf(\hat{A}_{\sigma} [y_1, \dots, y_d]/ \langle y_i^q-v_iy_{i+1} \rangle_{i \in \Z/(d+1)\Z}).\]
On a par l'identité  \ref{propliexpi} et \eqref{diagrammeuniv}  
\[\Lf_{\chi_{j}}/d_{j-1}(\Lf_{\chi_{j-1}}^{\otimes p})\cong \begin{cases}(T_{d+1-i})^{\vee}/\Pi((T_{d+2-i})^{\vee}) \cong \Of/x_{d+1-i} \Of & \si j=fi\in f\Z/f(d+1),\\ 0 &\sinon\end{cases} \] pour tout $j \in \Z/f(d+1)\Z$ ie. $v_i=x_{d-i}u_i$  avec $u_{i} \in \hat{A}_{\sigma}^*$. On  en déduit le résultat

Revenons à ce calcul de groupes de Picard. D'apr\`es \cite[3.7.4]{frvdp}, il suffit de raisonner en fibre sp\'eciale $\H_{\bar{\F},\sigma}^d=\spec (\bar{\F}[x_0,\cdots,x_d,\frac{1}{P_{\sigma}}]/(x_0 \cdots x_d))$. On la recouvre par les composantes irr\'eductibles $\H_{\bar{\F},\sigma}^d=\bigcup_{0\le i\le d} V(x_i)=:\bigcup_{0\le i\le d} V_i$. Nous souhaitons justifier que nous pouvons raisonner à la Cech sur ce recouvrement par des fermés. Cet argument est relativement similaire à celui réalisé dans \ref{coropic} et cela découle du résultat technique qui va suivre. 

Introduisons quelques notations nécessaires. Soit $E$ un ensemble fini, nous confondrons $\PC (E)$ ordonné par l'inclusion avec la catégorie associée. Pour tout foncteur covariant $I\subset E \mapsto A_I$ de $\PC (E)$ vers la catégorie des groupes abéliens, on considère le complexe de cochaîne suivant :
\[\CC^\bullet ((A_I)_{I\subset E}): 0 \to \bigoplus_{i\in E}A_{\{i\}}\to \bigoplus_{\{i,j\}\subset  E}A_{\{i,j\}}\to\cdots\to \bigoplus_{\substack{I\subset E\\ |I|=k+1}}A_{I} \to\cdots\]
les différentielles étant données par les sommes alternées des morphismes $r_{I,J}:A_I\to A_J$ (pour $I\subset J$) définis par fonctorialité.

\begin{ex}

Soit un recouvrement ouvert $\UC=\{U_i\}_{i\in E}$ d'un espace topologique $X$ et $\Ff$ un faisceau en groupes abéliens sur $X$. Dans ce cas, $\CC^\bullet ((\Ff(\bigcap_{i\in I}U_i))_{I\subset E})$ est le complexe de Cech $\check{\CC}^\bullet (X,\UC,\Ff)$.

\end{ex}

\begin{lem}\label{lemcomp}

Reprenons l'ensemble fini $E$ et le foncteur $(A_I,r_{I,J})_{I\subset J\subset E}$ comme précédemment. Supposons l'existence, pour tous $I\neq \emptyset$ et $J:=I\amalg \{j\}\subset E$, d'une section $\varphi_{I,J}:A_J\to A_I$ de $r_{I,J}$ qui vérifie pour tout $k\notin J$,
\begin{equation}\label{eq:sec}
\varphi_{I,J}(\ker (r_{J,J\cup \{k\}}))\subset \ker (r_{I,I\cup \{k\}})
\end{equation}
alors $\hhh^1 (\CC^\bullet ((A_I)_{I\subset E}))=0$.

\end{lem}
\begin{coro}\label{coroh1c}

On fixe un point fermé $\pG\in \H_{\bar{\F},\sigma}^d$, on a \[\hhh^1(\CC^\bullet ((\Of^*(V_I))_{I\subset E}))=\hhh^1(\CC^\bullet ((\Of^*_{V_I,\pG})_{I\subset E_{\pG}}))=0\]
avec 
$E=\left\llbracket 0,d\right\rrbracket$, $E_{\pG}=\{i\in E: \pG \in V_i\}$ et $V_I=\bigcap_{i\in I} V_i$ pour $I\subset E$. 
\end{coro}

\begin{proof}\eqref{lemcomp}
Pour établir le résultat, nous prouvons que pour un $1$-cocycle $c=(c_{i,j})_{\{i,j\}\subset E}$ qui s'annule pour un certain nombre de couples $(i,j)$, on peut trouver un cocycle cohomologue $c-\delta s$ qui s'annule sur plus de couples. Nous pourrons conclure par finitude de l'ensemble de définition de $c$. 

Donnons un sens plus précis à cette affirmation.  Pour une partie $F\subset E$, on dit qu'un $1$-cocycle $c=(c_{i,j})_{i,j}$ vérifie la propriété $\Pf_{F}$ si et seulement si pour tout $k\in F$ et $i\in E\backslash\{k\}$, on a $c_{i,k}=0$. En particulier, tous les cocycles vérifient $\Pf_{\emptyset}$ et un cocycle vérifiant  $\Pf_{E}$ est nul. Il suffit de montrer que pour toute partie $F\subset E$, tout $j\notin F$ et un cocycle $c$ vérifiant $\Pf_{F}$, on peut trouver un cocycle cohomologue $c-\delta s$ vérifiant la propriété  $\Pf_{F\cup\{j\}}$ ou de manière équivalente, $c-\delta s$ vérifie $\Pf_{F}$ et $\Pf_{\{j\}}$. Le résultat se déduira par récurrence immédiate.

Donnons-nous $F$, $j$, $c$ comme précédemment et construisons la $0$-cochaîne voulue $s=(s_i)_i$. Posons $s_j=0$ et $s_i=\varphi_{\{i\},\{i,j\}}(c_{i,j})\in A_{\{i\}}$. En particulier, si $k\in F$, $c_{k,j}=0$ d'où \[s_k=\varphi_{\{k\},\{k,j\}}(c_{k,j})=0.\] Montrons que $c-\delta s$ convient.

$\bullet$ Comme $s_j=0$, on a pour tout $i\in E$ \[(\delta s)_{i,j}=r_{\{i\},\{i,j\}}(s_i)=r_{\{i\},\{i,j\}}(\varphi_{\{i\},\{i,j\}}(c_{i,j}))=c_{i,j}.\]
Donc,  $c-\delta s$ vérifie  $\Pf_{\{j\}}$.

$\bullet$ Fixons  $k\in F$ et $i\in E\backslash\{k\}$. Comme $s_k=0$, on a  comme précédemment \[(\delta s)_{i,k}=r_{\{i\},\{i,k\}}(s_i)=r_{\{i\},\{i,k\}}(\varphi_{\{i\},\{i,j\}}(c_{i,j})).\] Pour avoir $(\delta s)_{i,k}=c_{i,k}=0$, il suffit de vérifier  $c_{i,j}\in\ker(r_{\{i,j\},\{i,j,k\}})$ d'après \eqref{eq:sec}. Mais   par fermeture de $c$,  on a \[r_{\{i,j\},\{i,j,k\}}(c_{i,j})=r_{\{i,k\},\{i,j,k\}}(c_{i,k})-r_{\{j,k\},\{i,j,k\}}(c_{j,k})=0\]  car  $c$ vérifie  $\Pf_{F}$. Ainsi, $(\delta s)_{i,k}=c_{i,k}=0$ et $c-\delta s$ vérifie  $\Pf_{F}$ ce qui termine la preuve.

\end{proof}



\begin{proof}\eqref{coroh1c}
Il suffit de construire des sections aux projections naturelles qui vérifient l'identité \eqref{eq:sec}.
Pour toute partie $I\subset E$ (resp. $I\subset E_{\pG}$), on peut trouver des parties multiplicatives $S^{(I)}$ (resp. $S^{(I)}_{\pG}$) pour lesquels on a $\Of(V_I)=(S^{(I)})^{-1}\bar{\F}[(x_i)_{i\notin I}]$ (resp. $\Of_{V_I,\pG}=(S^{(I)}_{\pG})^{-1}\bar{\F}[(x_i)_{i\notin I}]$). Si on a  $I\subset J\subset E$ (resp. $I\subset J\subset E_{\pG}$), l'endomorphisme de l'anneau $\bar{\F}[(x_i)_{i\notin I}]$ \[x_i\mapsto \begin{cases}x_i&\si i\notin J\\
0 &\sinon\end{cases}\] préserve les parties $S^{(I)}$ (resp. $S^{(I)}_{\pG}$)  et  se prolonge donc en un endomorphisme $\psi$ sur $\Of(V_I)$ (resp. $\psi_{\pG}$ sur $\Of_{V_I,\pG}$). La projection modulo l'idéal $\left\langle x_j \right\rangle_{j\in J\backslash I}$ induit un isomorphisme ${\rm Im}(\psi)\iso \Of(V_J)$ (resp. ${\rm Im}(\psi_{\pG})\iso \Of_{V_J,\pG}$) et on peut construire par composition une section à $\Of(V_I)\to \Of(V_J)$ (resp. $\Of_{V_I,\pG}\to \Of_{V_J,\pG}$):
\[\Of(V_J)\iso{\rm Im}(\psi)\flinj \Of(V_I)\ {\rm (resp.}\ \Of_{V_J,\pG}\iso{\rm Im}(\psi_{\pG})\flinj \Of_{V_I,\pG}).\]
Cette dernière envoie l'idéal $x_k\Of(V_J)$ (resp. $x_k\Of_{V_J,\pG}$) sur l'idéal $x_k \Of(V_J)$ (resp. $x_k \Of_{V_I,\pG})$) pour $k\notin J$. Ainsi, les sections obtenues au niveau des fonctions inversibles   vérifient l'identité \eqref{eq:sec} et on en déduit les annulations voulues.

\end{proof}

Nous voulons maintenant prouver l'exactitude de la suite de faisceau 
\begin{equation}\label{eq:cechf}
0 \to \Of^*_{\H_{\bar{\F},\sigma}^d}\fln{(\delta^0)^*}{} \bigoplus_{i\in E}\iota^* \Of^*_{V_i} \fln{(\delta^1)^*}{} \bigoplus_{\{i,j\}\subset  E}\iota^* \Of^*_{V_{\{i,j\}}}\fln{(\delta^2)^*}{}\bigoplus_{\{i,j,k\}\subset  E}\iota^* \Of^*_{V_{\{i,j,k\}}}.
\end{equation}
Pour l'exactitude en $\bigoplus_{\{i,j\}\subset  E}\iota^* \Of^*_{V_{\{i,j\}}}$, on peut raisonner en chaque point fermé $\pG\in \H_{\bar{\F},\sigma}^d$ et cela découle de l'annulation $\hhh^1(\CC^\bullet ((\Of^*_{V_I,\pG})_{I\subset E_{\pG}}))=0$. Pour le reste de la suite, on s'intéresse d'abord à l'exactitude de
\[0 \to \Of_{\H_{\bar{\F},\sigma}^d}\fln{(\delta^0)^+}{} \bigoplus_{i\in E}\iota^* \Of_{V_i} \fln{(\delta^1)^+}{} \bigoplus_{\{i,j\}\subset  E}\iota^* \Of_{V_{\{i,j\}}}.\]
Par cohérence, il suffit de l'établir au niveau des sections globales où cela résulte de l'identité $I(V_{i})+I(V_{j})=I(V_{\{i,j\}})$ pour tous $i,j$. En observant les égalités ensemblistes \[{\rm Im}((\delta^0)^*)={\rm Im}((\delta^0)^+)\cap \Of^*\et \ker((\delta^n)^*)=\ker((\delta^n)^+)\cap \Of^*\] pour $n=0,1$, on en déduit l'exactitude de \eqref{eq:cechf}. 

La suite exacte longue associée à $0 \to \Of^*_{\H_{\bar{\F},\sigma}^d}\to \bigoplus_{i\in E}\iota^* \Of^*_{V_i}\to\ker((\delta^2)^*)\to 0$ entraîne l'exactitude de
\[\bigoplus_{i\in E}\Of^*(V_i)\to\ker((\delta^2)^*) (\H_{\bar{\F},\sigma}^d)\to \pic(\H_{\bar{\F},\sigma}^d) \to \bigoplus_{i\in E}\pic(V_i).\]
D'après \ref{coroh1c}, on a \[\ker((\delta^2)^*) (\H_{\bar{\F},\sigma}^d)/(\delta^1)^*(\bigoplus_{i\in E}\Of^*(V_i))=\hhh^1(\CC^\bullet ((\Of^*_{V_I})_{I\subset E}))=0.\] De plus, $\pic(V_i)=0$ car $V_i$ est le spectre d'une $\bar{\F}$-algèbre factorielle de type fini d'où $\pic(\H_{\bar{\F},\sigma}^d)=0$.



\subsection{Preuve de \ref{theoxpid} \eqref{equi} et conséquences}

Nous donnons maintenant l'expression des fonctions $(u_i)_i$ caractérisant le schéma en groupe $\XG[\Pi_D]_{\sigma}$. Cette famille d'éléments est unique modulo une relation d'équivalence décrite dans \ref{propxpidlib}. Un des points essentiels est de comprendre comment déterminer cette classe équivalence à partir de la géométrie du revêtement $\Sigma^1$. Pour cela, nous associons deux fonctions à toute famille de fonctions inversibles $(v_i)_i$,

\[V((v_i)_i)=v_0 v_{d}^q v_{d-1}^{q^2}\dots v_1^{q^d},\]\[\tilde{V}((v_i)_i)=x_0^{q-1} \dots x_{d-1}^{q^d-1}V((v_i)_i).\]

Nous pourront caractériser les classes d'équivalence recherchées grâce aux fonctions $V((u_i)_i)$, $\tilde{V}((u_i)_i)$. Plus précisément, les fonctions  $(u_i)_i$ sont uniques modulo la relation d'\'equivalence $\sim$ suivante 
\begin{align*}
(x_{d-i}u_i) \sim (x_{d-i} \tilde{u}_i) & \text{ ssi } \exists (w_i)_i \in \hat{A}_{\sigma}^*, \; \frac{u_i}{\tilde{u}_i}= \frac{w_i^q}{w_{i+1}} \\
& \text{ ssi } \exists w_1 \in \hat{A}_{\sigma}^*, \; w_1^NV((u_i)_i)= V((\tilde{u}_i)_i) \\ 
& \text{ ssi } \tilde{V}((u_i)_i) = \tilde{V}((\tilde{u}_i)_i)\pmod{({A}_{\sigma}^*)^N}
\end{align*}
avec $N=q^{d+1}-1=(q-1) \tilde{N}$. Dit autrement, on a des inclusions induites par $\tilde{V}$ et le morphisme de bord dans la suite exacte longue de Kummer $\kappa$


\[(\prod_{i \in \Z/(d+1)\Z} \hat{A}_{\sigma}^*)/\sim\  \flinj A_\sigma/(A_\sigma)^N\flinj\het{1}(\H_{ \breve{K}, \sigma}^d,\mu_N).\] 
Cette composition de flèches $(u_i)_i\fla \kappa(\tilde{V}((u_i)_i))$ a une interprétation géométrique ce qui nous permet de déterminer la famille $(u_i)_i$ grâce à la géométrie de $\Sigma^1_{\sigma}$. Plus précisément, on a \[ B_{\sigma}:= \hat{B}_{\sigma}[\frac{1}{p}]=A_{\sigma}[y_1]/ (y_1^{q^{d+1}}- \varpi \tilde{V}((u_i)_i)y_1).\]
Comme $\Sigma^1_{\sigma}= (\XG[\Pi_D]_{\sigma} \setminus \{0\})^{rig}$, on a \[ \Sigma^1_{\sigma}=\spg (B_\sigma [\frac{1}{y_1}])= \H_{ \breve{K}, \sigma}^d(( \varpi  \tilde{V}((u_i)_i))^{1/N}).\]  D'après cette observation, il suffit  de prouver

\begin{lem}\label{lemeqsigsig}

Pour un simplexe maximal $\sigma$ et $N=q^{d+1}-1$, on a l'identité
\[\Sigma^1_\sigma=\H_{ \breve{K}, \sigma}^d(( \varpi \tilde{V}((\frac{\prod_{\tilde{b}\in R_{d-1-i}} P_{\tilde{b}}}{ \prod_{\tilde{b}\in R_{d-i}} P_{\tilde{b}}} )_i))^{1/N}).\]
\end{lem}

\begin{proof}
On a $\H_{  \breve{K}, \sigma}^d\subset \bar{U}_{1,\breve{K}}$ et d'après \ref{theoeqsigma1}, $\Sigma^1|_{\bar{U}_{1,\breve{K}}}=\bar{U}_{1,\breve{K}}(V^{1/N})$ avec \[V=\varpi \prod_{a \in \P^d(\OC_K/\varpi^2)} (\frac{l_{\tilde{a}}(z)}{z_d})^{q(q-1)}  \]
pour $z=[z_0, \dots, z_d]$   des coordonn\'ees homog\`enes de $\P^d_{ K}$ adapt\'ees \`a $\sigma$. Prouver le résultat revient à montrer la congruence suivante :
\[ V \equiv \varpi \tilde{V}((\frac{\prod_{\tilde{b}\in R_{d-1-i}} P_{\tilde{b}}}{ \prod_{\tilde{b}\in R_{d-i}} P_{\tilde{b}}})_i)\pmod{ (A_{\sigma}^*)^N}.\]
Pour cela, réécrivons chaque $\frac{l_{\tilde{a}}(z)}{z_d}$ en fonction des variables $(x_i)_i$ de $\hat{A}_\sigma$. Nous aurons besoin de 
quelques conventions supplémentaires. 

 Tout hyperplan $a \in \P^d(\OC_K/\varpi^2)$ s'écrit comme le noyau d'une forme linéaire $\ker(l_{\tilde{a}})$ pour un unique vecteur $\tilde{a}\in\OC_K^{d+1}$ de la forme :
\[\tilde{a}=( \tilde{b}_{0}, \dots,  \tilde{b}_{i(a)-1}, 1, \varpi\tilde{b}_{i(a)+1},  \dots, \varpi\tilde{b}_{d})+(\varpi\tilde{c}_{0}, \dots,  \varpi\tilde{c}_{i(a)-1},0, \dots, 0)=\tilde{b}+\varpi\tilde{c}\]
avec  $\tilde{b}_j,\tilde{c}_j\in\mu_{q-1}(K)\cup\{0\}$ pour tout $j$ et $i(a)$ l'unique entier $i$ tel que $\tilde{a} \in M_{i} \setminus M_{i-1}$. 

Quand $a$ parcourt les éléments de $\P^d(\OC_K/\varpi^2)$ tel que    $i(a)=i_0$ pour $i_0$ fixé, les vecteurs $\tilde{b}$ décrivent un système de représentants $R_{i_0}$ défini dans \ref{paragraphbtgeosimpstd}. Dit autrement, la flèche :
\[(\tilde{b},c)\in\amalg_i (R_i \times \F^i)\fla [\tilde{b}+\varpi\tilde{c}]\in \P^d(\OC_K/\varpi^2)\] est une bijection qui envoie $ R_i \times \F^i$ sur $\{a \in \P^d(\OC_K/\varpi^2):i(a)=i\}$. 

Revenons à la description des fractions rationnelles $\frac{l_{\tilde{a}}(z)}{z_d}$ pour $\tilde{a}=\tilde{b}+\varpi\tilde{c}$. On rappelle les relations $x_d= \varpi \frac{z_d}{z_0}$, $x_i= \frac{z_i}{z_{i+1}}$ et $\frac{l_{\tilde{b}}}{z_{i}}=P_{\tilde{b}}(x)$ quand $\tilde{b}\in R_i$. On en déduit

\[ \frac{l_{\tilde{a}}}{z_d}\equiv\frac{z_{i(a)}}{z_d}\frac{l_{\tilde{b}}}{z_{i(a)}}\equiv (x_{i(a)} \dots x_{d-1})P_{\tilde{b}}(x) \pmod{1+\varpi\hat{A}_\sigma}. \] 
Ainsi, (on rappelle l'inclusion $1+\varpi\hat{A}_\sigma\subset  (A_\sigma^*)^N$) 
\begin{align*}
V & \equiv\varpi\prod_{a \in \P^d(\OC_K/\varpi^2)} [(x_{i(a)}\dots x_{d-1})P_{\tilde{b}}(x)]^{q(q-1)}\equiv  \varpi \prod_i \prod_{\substack{\tilde{b} \in R_{i} \\ (c_j) \in \F^{i}}} [ (x_{i} \dots x_{d-1})P_{\tilde{b}}(x))]^{q(q-1)} \\ 
& \equiv \varpi \prod_i x_i^{q^{d+2+i}-q^{d+1}} \prod_{\tilde{b}\in R_{i}} P_{\tilde{b}}^{q^{i+1}(q-1)}\equiv \varpi \prod_i x_i^{q^{i+1}-1} \left(\frac{\prod_{\tilde{b}\in R_{d-1-i}} P_{\tilde{b}}}{ \prod_{\tilde{b}\in R_{d-i}} P_{\tilde{b}}}  \right)^{q^{d+1-i}} \pmod{(A_\sigma^*)^N}.
\end{align*}
La dernière expression est exactement la quantité $\varpi \tilde{V}((\frac{\prod_{\tilde{b}\in R_{d-1-i}} P_{\tilde{b}}}{ \prod_{\tilde{b}\in R_{d-i}} P_{\tilde{b}}} )_i)$ ce qui termine la preuve.
\end{proof}

Nous terminons cette discussion par cette conséquence remarquable.

\begin{coro}
\label{corol0l1triv}
$\Lf_0 \otimes \dots \otimes \Lf_d$ est un fibré trivial sur $\H_{  \OC_{\breve{K}}}^d$.  
\end{coro}

\begin{proof}

Intéressons-nous d'abord à la fibre générique. On a encore un schéma de Raynaud $\XG[\Pi_D]^{rig}$ sur $\H_{  {\breve{K}}}^d$ qui est libre d'après \ref{theocohoanhdk} point 1. Il admet la décomposition en partie isotypique $\pi_*\Of_{\XG[\Pi_D]^{rig}}=\Of_{\H_{ {\breve{K}}}^d}\oplus \bigoplus_{\chi} \Lf_{\chi}[\frac{1}{\varpi}]$ où $\Lf_{\chi}$ sont les parties isotypiques de l'idéal d'augmentation  du schéma de Raynaud $\XG[\Pi_D]$ et $\Lf_{\chi}[\frac{1}{\varpi}]$ sont les fibrés induits en fibre générique. De plus, d'après \ref{remlierep} et \ref{propliexpi}, le morphisme induit par la multiplication $d_i :\Lf_{\chi_i}[\frac{1}{\varpi}]^{\otimes p}\iso \Lf_{\chi_{i+1}}[\frac{1}{\varpi}]$ est un isomorphisme de fibré en droite. On en déduit l'écriture d'après corollaire \ref{cororaynlib} et la note \eqref{foouni}
\[\XG[\Pi_D]^{rig}=\underline{\spg}_{\H_{ {\breve{K}}}^d}(\Of_{\H_{ {\breve{K}}}^d}[Y_1]/(Y_1^{q^d}-wY_1))\]
avec $w\in \Of^*(\H_{ {\breve{K}}}^d)$ et $Y_1$ un générateur de $\Lf_{1}[\frac{1}{\varpi}]$. Mais, on observe
\[\Sigma^1=\underline{\spg}_{\XG[\Pi_D]^{rig}}(\Of_{\XG[\Pi_D]^{rig}}[1/Y_1])=\H_{ {\breve{K}}}^d(w^{1/N})\]
et  la fonction inversible $w$ s'écrit sous la forme $\varpi u^{q-1}$ avec $u\in \Of^*(\H_{ {\breve{K}}}^d)$ d'après \ref{theoeqsigma1}. On introduit alors $t_0:=\frac{Y_1^{\tilde{N}}}{u}$. C'est un élément de $\Lf_{1}[\frac{1}{\varpi}]^{\otimes \tilde{N}}(\H_{  {\breve{K}}}^d)\cong (\Lf_0 \otimes \dots \otimes \Lf_d)[\frac{1}{\varpi}](\H_{ {\breve{K}}}^d)$ qui vérifie l'identité $t_0^q=\varpi t_0$. Il s'agit de voir que $t_0$ engendre $\Lf_0 \otimes \dots \otimes \Lf_d$.

Pla\c{c}ons-nous sur un simplexe maximal $\sigma$ et reprenons les donn\'ees $y_i$, $u_i$ de \ref{theoxpid}. Au-dessus de $\sigma$, $Z_{\sigma}=y_0 \dots y_d$ est un g\'en\'erateur de $\Lf_0 \otimes \dots \otimes \Lf_d$ qui v\'erifie d'après \eqref{equi} et \ref{remui} point 2.
\begin{align*}
Z_{\sigma}^q& = (u_0 x_1) (u_1 x_2) \dots (u_d x_0) Z_{\sigma} \\
&= \varpi u_0 \dots u_d Z_{\sigma}= \varpi Z_{\sigma}.
\end{align*}
\'Etudions les solutions de $Z^q= \varpi Z$ dans $\Lf_0 \otimes \dots \otimes \Lf_d[\frac{1}{\varpi}](\H_{ {\breve{K}}, \sigma}^d)\cong Z_{\sigma}\Of(\H_{  {\breve{K}},\sigma}^d)$ ie. celles de la forme  $Z_{\sigma} f$ avec $f\in A_\sigma$. On a \[\varpi Z_{\sigma} f=Z_{\sigma}^q f^q=\varpi Z_{\sigma} f^q\] et $f=f^q$ par unicité de l'écriture d'où $f \in \mu_{q-1}(\breve{K})\cup \{0\}$ ($\H_{  {\breve{K}}, \sigma}^d$ est géométriquement connexe). Ainsi, toutes les solutions non-nulles engendrent le faisceau ${\Lf_0 \otimes \dots \otimes \Lf_d}_{|\H_{  \OC_{\breve{K}}, \sigma}^d}$ et $t_0$ est donc un générateur global du faisceau $\Lf_0 \otimes \dots \otimes \Lf_d$.



\end{proof}

\begin{rem}

Ce résultat induit une décomposition de l'idéal d'augmentation $\IC=\bigoplus_{i=0}^{q-2}\IC_i$ où les termes \[\IC_i:=(\Lf_0 \otimes \dots \otimes \Lf_d)^{\otimes i} \bigoplus_{{\substack{(\alpha_j)_{j}\in \left\llbracket 0,q-1\right\rrbracket^{d}}}}\Lf^{{\otimes \alpha_1}}_1\otimes  \cdots\otimes \Lf^{{\otimes \alpha_d}}_d\] sont isomorphes deux à deux. C'est une manifestation en niveau entier du faite que $\Sigma^1$ admettent $q-1$ composantes connexes géométriques. En effet, $\Lf_0 \otimes \dots \otimes \Lf_d$ est engendré par une solution de l'équation $Z^q= \varpi Z$ et les racines du polynôme $Z^{q-1}= \varpi $ sont en bijection avec $\pi_0^{geo}(\Sigma^1)$.

\end{rem}

\newpage

\bibliographystyle{alpha}
\bibliography{eqsig_v1}

\end{document}